\documentclass{scrartcl}

\usepackage{graphicx}
\usepackage{pgf}
\usepackage{color}
\usepackage[colorlinks=true]{hyperref}
\usepackage[left=2.5cm,right=2.5cm, top=1.5cm, vmargin=2.5cm,footnotesep=0.5cm]{geometry}
\usepackage{amsmath}
\usepackage{amsfonts}
\usepackage{amssymb}
\usepackage{amsthm}

\newtheorem{example}{Example}[section]
\newtheorem{corollary}{Corollary}[section]
\newtheorem{proposition}{Proposition}[section]
\newtheorem{lemma}{Lemma}[section]
\newtheorem{remark}{Remark}[section]
\newtheorem{definition}{Definition}[section]
\newtheorem{theorem}{Theorem}[section]

\newcommand{\RR}{\mathbb{R}}

\newcommand{\NN}{\mathbb{N}}

\DeclareMathOperator{\cl}{cl}
\DeclareMathOperator{\conv}{conv}
\DeclareMathOperator{\cor}{cor}
\DeclareMathOperator{\intt}{int}
\DeclareMathOperator{\bd}{bd}

\DeclareMathOperator{\sepC}{\mathfrak{C}}

\makeatletter
\DeclareOldFontCommand{\rm}{\normalfont\rmfamily}{\mathrm}
\DeclareOldFontCommand{\sf}{\normalfont\sffamily}{\mathsf}
\DeclareOldFontCommand{\tt}{\normalfont\ttfamily}{\mathtt}
\DeclareOldFontCommand{\bf}{\normalfont\bfseries}{\mathbf}
\DeclareOldFontCommand{\it}{\normalfont\itshape}{\mathit}
\DeclareOldFontCommand{\sl}{\normalfont\slshape}{\@nomath\sl}
\DeclareOldFontCommand{\sc}{\normalfont\scshape}{\@nomath\sc}
\makeatother
\addtokomafont{disposition}{\rmfamily}

\begin{document}

\title{
Symmetric and Non-Symmetric Cone Separation via Bishop-Phelps Cones in Normed Spaces
}

\author{
Fernando Garc\'ia-Casta\~no\thanks{
              Departamento de Matemáticas, Universidad of Alicante;
              Alicante, Spain;
              Email: fernando.gc@ua.es;
              Orcid: 0000-0002-8352-8235 (Corresponding Author)}
           \and
            Christian G\"unther\thanks{
            Leibniz Universit\"at Hannover, Institut f\"ur Angewandte Mathematik;
            30167 Hannover, Germany; Email: c.guenther@ifam.uni-hannover.de; Orcid: 0000-0002-1491-4896}
           \and
           M.A. Melguizo Padial\thanks{
           Departamento de Matemáticas, Universidad of Alicante;
              Alicante, Spain;
              Email: ma.mp@ua.es;
              Orcid: 0000-0003-0303-791X}
            \and
            Christiane Tammer\thanks{
            Martin Luther University Halle-Wittenberg, Faculty of Natural Sciences II, Institute of Mathematics;
            06099 Halle (Saale), Germany;
            Email: Christiane.Tammer@mathematik.uni-halle.de}
            \\[0.8cm]
\large{Dedicated to Professor Qamrul Hasan Ansari on the occasion of his 65th birthday}
}

\maketitle

\begin{center}\vspace{-1cm}
	\textbf{Abstract}
\end{center}

\begin{abstract}
	In this paper, we study relationships between symmetric and non-symmetric separation of (not necessarily convex) cones by using separating cones of Bishop-Phelps type in real normed spaces. Besides extending some known results for the non-symmetric cone separation approach, we propose a new symmetric cone separation approach and establish cone separation results for it by using some cone separation results obtained for the non-symmetric cone separation approach twice (by swapping the roles of the cones). In addition to specifically emphasizing the results for the convex case, we also present some existence results for (bounded) convex bases of convex cones. Finally, we highlight some applications of symmetric and non-symmetric cone separation in optimization.
\end{abstract}

\begin{flushleft}
\end{flushleft}

\begin{flushleft}
	\textbf{Mathematics subject classifications (MSC 2010):} 90C29, 90C25, 90C31, 90C48, 46N10
\end{flushleft}

\section{Introduction} \label{sec:introduction}

It is well-known that separation plays an important role in variational analysis and optimization.
In particular, \textbf{cone separation theorems} (i.e., theorems related to the separation of two cones by a hyperplane or a certain conical surface) have been studied by several authors in the literature (see, e.g., \cite{GarciaCastano2023},  \cite[Sec. 2.4.3]{GoeRiaTamZal2023}, \cite{gunther2023nonlinear,gunther2023nonlinear_reflexive}, \cite{Henig}, \cite[Th. 3.22]{Jahn2011}, \cite[Sec. 3.7]{Jahn2023}, \cite[Th. 4.3]{Kasimbeyli2010}, \cite[Sec. 2.3]{Khazayel2021a},  \cite{Nieuwenhuis},  \cite[Cor. 2.3]{NovoZalinescu2024}, \cite[Sec. 11]{Rockafellar1970}, and \cite[Sec. 4]{Soltan2018}). Such cone separation theorems are known to be useful in optimization (e.g., for deriving scalarization results for nonconvex vector optimization problems; see, e.g., \cite{AKY2018}, \cite{GarciaCastano2023}, \cite{Gerstewitz1983}, \cite{GerthWeidner}, \cite{gunther2024scalarization}, \cite{Jahn2023}, \cite{Kasimbeyli2010}, \cite{Kasimbeyli2013}, \cite{TammerWeidner2020}). 
In the literature, there are different concepts for the nonlinear separation of two (not necessarily convex) cones by a cone/conical surface (see, e.g., \cite{gunther2023nonlinear,gunther2023nonlinear_reflexive},  \cite{Henig},  \cite{Kasimbeyli2010}, \cite{Nehse1981},  \cite{Nieuwenhuis}).  
Let us formalize our underlying cone separation concepts below.

\medskip

Consider two non-trivial cones $C, K\subseteq X$ (that is, cones different from $\{0_X\}$ and $X$ itself) and a closed, solid, convex cone $\sepC \subseteq X$ (in what follows $\intt\, \sepC$, $\bd\, \sepC$,  $\cl\, \sepC$ stand for the interior, boundary, and closure of $\sepC$, respectively) in a real normed space $(X, ||\cdot||)$. 
\begin{definition} \label{def:separation_cones_with_C}
We say that the cones $C$ and $K$ are 
\begin{itemize}
    \item[$\bullet$] \textbf{strictly separated by (the boundary of the cone) $\sepC$ in a non-symmetric way} (for short, $\sepC \in \mathcal{N}(C,K)$) if 
    \begin{equation}
        \label{eq:sep_1}
        C \setminus \{0_X\} \subseteq \intt\, \sepC \quad \mbox{and} \quad K \setminus \{0_X\} \subseteq X \setminus \sepC,
    \end{equation}
    or equivalently,
    \begin{equation*}
        \label{eq:sep_2}
    C \setminus \{0_X\} \subseteq \intt\, \sepC \quad \mbox{and} \quad K \cap \sepC = \{0_X\}.
    \end{equation*}
    \item[$\bullet$] \textbf{strictly separated by  $\sepC$ (in a symmetric way)} if $$\sepC \in \mathcal{S}(C,K) := \mathcal{N}(C,K) \cup \mathcal{N}(K,C),$$
    where the set $\mathcal{N}(C,K)$ (respectively, $\mathcal{S}(C,K)$) consists of all closed, solid, convex cones in $X$ which are non-symmetric (respectively, symmetric) strictly  separating cones for $C$ and $K$.
\end{itemize}
\end{definition}

\begin{remark} The non-symmetric strict cone separation approach from Definition \ref{def:separation_cones_with_C} is already used in \cite{GarciaCastano2023}, \cite{gunther2023nonlinear,gunther2023nonlinear_reflexive}, \cite{Henig}, \cite{Kasimbeyli2010}, while the symmetric strict cone separation approach appears to be new in the literature to our knowledge.
Note that we have $-\mathcal{N}(C,K) = \mathcal{N}(-C, -K)$.
\end{remark}

In our upcoming cone separation theorems, we like to consider a closed, solid and convex separating cone $\sepC$ that can be represented by the sublevel set (w.r.t. the level 0) of a (lower semicontinuous, convex, positively homogeneous) function $\varphi: X \to \mathbb{R}$. More precisely, $\varphi$ should satisfy the cone representation properties
$$
\sepC = \{x \in X \mid \varphi(x) \leq 0 \} \quad \mbox{and} \quad \intt \sepC = \{x \in X \mid \varphi(x) < 0 \}.
$$
In this case, the condition \eqref{eq:sep_1} (given in an order theory way) can be written in terms of functional analysis
$$
\varphi(x^1) < 0 < \varphi(x^2) \quad \text{for all }x^1 \in C \setminus \{0_X\} \text{ and }  x^2 \in K\setminus \{0_X\}.
$$

In this paper, we study relationships between \textbf{symmetric cone separation} and \textbf{non-symmetric cone separation} of (not necessarily convex) cones by using separating cones of Bishop-Phelps type in real normed spaces.
While the non-symmetric strict separation approach for cones is well studied in the literature (see  \cite{GarciaCastano2023}, \cite{gunther2023nonlinear,gunther2023nonlinear_reflexive}, \cite{Kasimbeyli2010}), the symmetric strict separation approach based on Bishop-Phelps separating cones is new.

For any $(x^*, \alpha) \in X^* \times \RR$ (where $X^*$ is the dual normed space of $X$), let us define two closed cones
\begin{align*}
    C(x^*,\alpha) & :=\{x\in X\mid x^*(x)\geq \alpha \|x\|\},\\
    S(x^*,\alpha) & := -C(x^*,\alpha),
\end{align*}
two sets
\begin{align*}
C^>(x^*,\alpha) & :=\{x\in X\mid x^*(x)>\alpha \|x\|\}\quad (\subseteq C(x^*,\alpha)),\\
    S^<(x^*,\alpha) & := -C^>(x^*,\alpha)\quad (\subseteq S(x^*,\alpha)),
\end{align*}
and a so-called \textbf{norm-linear function} $\varphi_{x^*, \alpha}: X \to \mathbb{R}$ by 
$$\varphi_{x^*, \alpha}(x) := x^*(x) + \alpha ||x|| \quad \text{for all }x \in X.$$
Note that 
\begin{align*}
    -C(x^*,\alpha) = S(x^*,\alpha) = \{x \in X \mid \varphi_{x^*, \alpha}(x) \leq 0\},\\
    -C^>(x^*,\alpha) = S^<(x^*,\alpha) = \{x \in X \mid \varphi_{x^*, \alpha}(x) < 0\}.
\end{align*}
Throughout, we denote the set of non-negative real numbers by $\mathbb{R}_+$ and the set of positive real numbers by $\mathbb{P}$.
\begin{remark} \label{rem:BPcone}
In the case that $(x^*, \alpha) \in (X^* \setminus \{0_{X^*}\}) \times \mathbb{P}$, the set $C(x^*,\alpha)$ is known as \textbf{Bishop-Phelps cone} and the function $\varphi_{x^*, \alpha}$ as \textbf{Bishop-Phelps function} (named after the work by Bishop and Phelps \cite{BP1962}). 
It is known that Bishop-Phelps cones / functions have a lot of useful properties and there are interesting applications in variational analysis and optimization (see, e.g., \cite{GarciaCastano2023}, \cite{gunther2023nonlinear,gunther2023nonlinear_reflexive}, \cite{HaJahn2017,HaJahn2021}, \cite{Ha2022}, \cite{Jahn2009,Jahn2023}, \cite{Kasimbeyli2010,Kasimbeyli2013}). 
Any Bishop-Phelps cone $C(x^*,\alpha)$ is a closed, pointed, convex cone. If $||x^*||_* > \alpha$ (where $||\cdot||_*: X^* \to \mathbb{R}$ denotes the dual norm of $||\cdot||$), then $C(x^*,\alpha)$ is non-trivial and 
$$C^>(x^*,\alpha) = \intt\, C(x^*,\alpha) \neq \emptyset \quad \text{ and } \quad (\bd\,C(x^*,\alpha))\setminus \{0_X\} \neq \emptyset.$$
Note that for $(x^*, \alpha) \in X^* \times \mathbb{P}$ with $||x^*||_* <  \alpha$ we have $C(x^*,\alpha) = \{0_X\}$, while for $||x^*||_* \leq  \alpha$ we have $C^>(x^*,\alpha) = \emptyset$.
If we consider $\alpha \in \mathbb{R}$ in the Bishop-Phelps function $\varphi_{x^*, \alpha}$ instead of $\alpha \in \mathbb{P}$, we call $\varphi_{x^*, \alpha}$ a norm-linear function (as proposed by Zaffaroni \cite{Zaffaroni2022}).
\end{remark}

In the separation of two (not necessarily convex) cones in a real normed space $X$, we are interested in a Bishop-Phelps cone $\sepC = C(x^*,\alpha)$ that strictly separates  (respectively,  in a non-symmetric way) two cones $C$ and $K$,  i.e., $\sepC \in \mathcal{S}(C,K)$ (respectively, $\sepC \in \mathcal{N}(C,K)$).

From now on, we will use calligraphic uppercase letters to denote families of subsets of $X$.
Consider a family $\mathcal{C}$ of cones in $X$ with
$$
\mathcal{C} \subseteq \{\sepC \subseteq X \mid \sepC \text{ is a closed, solid, convex cone}\} =: \mathcal{C}_{{\rm convex}}.
$$

\begin{definition} \label{def:separation_cones_with_C_family}
We say that the cones $C$ and $K$ in $X$ are
\begin{itemize}
    \item[$\bullet$]  \textbf{strictly separated w.r.t. $\mathcal{C}$ in a non-symmetric way} if 
    $$
        \mathcal{N}(C,K \mid \mathcal{C}) := \mathcal{C} \cap \mathcal{N}(C,K) \neq \emptyset.
    $$
    \item[$\bullet$] \textbf{strictly separated w.r.t. $\mathcal{C}$} if 
    $$
        \mathcal{S}(C,K \mid \mathcal{C}) := \mathcal{C} \cap \mathcal{S}(C,K) \neq \emptyset.
    $$
\end{itemize}
\end{definition}
\begin{remark}
It is easy to check that 
\begin{align*}
    -\mathcal{N}(C,K \mid \mathcal{C}) 
    & = \mathcal{N}(-C, -K \mid \mathcal{C});\\
    \mathcal{N}(C,K \mid \mathcal{C}) & \subseteq \mathcal{S}(C,K \mid \mathcal{C});\\
    \mathcal{S}(K,C \mid \mathcal{C}) & = \mathcal{S}(C,K \mid \mathcal{C})  = \mathcal{N}(C,K \mid \mathcal{C}) \cup \mathcal{N}(K,C \mid \mathcal{C}).
\end{align*}
Note that the case 
$\mathcal{N}(C,K \mid \mathcal{C}) \neq \emptyset = \mathcal{N}(K,C \mid \mathcal{C})$
might happen (see e.g. \cite[Sec. 4]{gunther2023nonlinear_reflexive}). 
\end{remark}

In this paper, we will focus on the separation of the given cones $C$ and $K$ by separating cones that belong to the following families of cones:
\begin{align*}
   \mathcal{C}_{{\rm BP}} & := \{C(x^*, \alpha) \mid (x^*, \alpha) \in X^* \times \mathbb{P},\, ||x^*||_* > \alpha \},\\
   \mathcal{C}_{{\rm Lin}} & := \{C(x^*, 0) \mid x^* \in X^*\setminus \{0_{X^*}\}\},
\end{align*}
and other subfamilies $\mathcal{C}_{{\rm BP}^*} \subseteq \mathcal{C}_{{\rm BP}}$ that will be introduced in Section \ref{sec:sym_and_non-sym_cone_separation}. Note that $\mathcal{C}_{{\rm Lin}} \cup \mathcal{C}_{{\rm BP}} \subseteq   \mathcal{C}_{{\rm convex}}$. 

\begin{remark}
    It is easy to check that 
    $$\mathcal{N}(K, C \mid \mathcal{C}_{{\rm Lin}}) = -\mathcal{N}(C, K \mid \mathcal{C}_{{\rm Lin}})$$ 
    and
    $$\mathcal{S}(C,K \mid \mathcal{C}_{{\rm Lin}}) = (\mathcal{N}(C,K \mid \mathcal{C}_{{\rm Lin}})) \cup (-\mathcal{N}(C,K \mid \mathcal{C}_{{\rm Lin}})).$$

The non-symmetric cone separation concept is known to be useful in vector optimization, e.g. in the proper efficiency solution concept in the sense of Henig, one needs to have dilating / enlargement convex cones of the given ordering cone (note that the convex cone $\sepC$ is said to be a dilating cone for the cone $C$ if  $C\setminus \{0_X\} \subseteq \intt\, \sepC$). It may also be interesting that such dilating / enlargement cones belong to a special family $\mathcal{C}$ of convex cones (e.g. $\sepC \in \mathcal{C}_{{\rm BP}}$).

From the viewpoint of classical separation in convex analysis and optimization (not considering a specific application), a symmetric separation concept for cones seems to be preferable. 
\end{remark}

On the one hand, the aim of the paper is to give characterizations for the conditions 
\begin{align*}
    \mathcal{N}(C,K \mid \mathcal{C}_{{\rm BP}^*}) & \neq \emptyset;\\
    \mathcal{N}(C,K \mid \mathcal{C}_{{\rm BP}^*})& \neq \emptyset \neq \mathcal{N}(K,C \mid \mathcal{C}_{{\rm BP}^*});\\
    \mathcal{S}(C,K \mid \mathcal{C}_{{\rm BP}^*}) &\neq \emptyset.
\end{align*}
for certain families $\mathcal{C}_{{\rm BP}^*} \subseteq \mathcal{C}_{{\rm BP}}$ and under certain assumptions on the cones $C$ and $K$.
On the other hand, we are interested in the same conditions with $\cl\,C$ or $\bd\,C$ in the role of $C$, as well as $\cl\,K$ or $\bd\,K$ in the role of $K$. Furthermore, the corresponding relationships between such conditions are of interest. Note that $\bd\,K = \bd((X \setminus K)\cup \{0_X\})$, and, if $C \subseteq K$, then $C \cap ((X \setminus K)\cup \{0_X\}) = \{0_X\}$. In this case, we are interested in separating cones $\sepC$ that belong to the sets $\mathcal{N}(C, (X \setminus K) \cup \{0_X\} \mid \mathcal{C}_{{\rm BP}^*})$ and $\mathcal{S}(C, (X \setminus K)\cup \{0_X\} \mid \mathcal{C}_{{\rm BP}^*})$, respectively, $\mathcal{N}(C, \bd\,K \mid \mathcal{C}_{{\rm BP}^*})$ and $\mathcal{S}(C, \bd\, K \mid \mathcal{C}_{{\rm BP}^*})$. This allows us to derive novel sufficient conditions (see Theorems \ref{th:Omega1and2_BP_cone_nonsym} and \ref{th:Omega1and2_BP_cone_sym}) that ensure
\begin{itemize}
    \item[$\bullet$] $C \cap K = \{0_X\}$ \\$\Longrightarrow \sepC \cap K = \{0_X\}$ for some $\sepC \in \mathcal{C}_{{\rm BP}^*}$ with $C \subseteq \sepC$
    (as studied in \cite[Theorem 5.2]{Kasimbeyli2010}, \cite[Theorem 2.4]{Jahn1985} for $\sepC \in \mathcal{C}_{{\rm convex}}$);
    \item[$\bullet$] $C \cap K = \{0_X\}$\\$\Longrightarrow \sepC \cap  C = \{0_X\} \text{ or } K \cap \sepC = \{0_X\}$ for some $\sepC \in \mathcal{C}_{{\rm BP}^*}$ with $C \subseteq \sepC$ or $K \subseteq \sepC$;
    \item[$\bullet$] $C \subseteq K$\\
    $\Longrightarrow  \intt\,C \subseteq \intt\,\sepC \subseteq \intt\,K \text{ and } C \subseteq \sepC \subseteq K$ for some $\sepC \in \mathcal{C}_{{\rm BP}^*}$. 
\end{itemize}
Our results extend known results derived in  \cite{GarciaCastano2023}, \cite{gunther2023nonlinear} and \cite{Kasimbeyli2010}.

\medskip

The paper is structured as follows.
In Section \ref{sec:preliminaries} we present some basics in real normed spaces and study cones and their properties, separation results for convex sets / cones (in particular,  Proposition \ref{prop:basic separation} is a key result for our work), and augmented dual cones of cones. We also recall known nonlinear separation results for cones based on the non-symmetric approach.
In Section~\ref{sec:sym_and_non-sym_separation} we present our main separation conditions for symmetric and non-symmetric separation of two given cones and study their basic relations. In particular, Propositions \ref{prop:basic separation_BP_e} and \ref{prop:basic separation_BP_f} are key results to connect the symmetric cone separation approach with the non-symmetric cone separation approach.
Section \ref{sec:sym_and_non-sym_cone_separation}  is devoted to symmetric and non-symmetric cone separation based on Bishop-Phelps separating cones.  The main aims of the paper, which we have described above, are achieved in this section. 
Moreover, in Section \ref{sec: existence_bases} we derive some new existence results for bases of convex cones.
Finally, in Section \ref{sec:conclusions} we present a conclusion with some applications in optimization and an outlook for further research.

\section{Preliminaries in Real Normed Spaces} \label{sec:preliminaries}

\subsection{Topological Basics}

Throughout the paper, assume that $X$ is a real normed space endowed with the norm $||\cdot||: X \to \mathbb{R}$, $X^*$ is its topological dual space endowed with the dual norm $||\cdot||_*$,  $\mathbb{S}_X$ is the unit sphere of $X$, $\mathbb{B}_X$ is the closed unit ball of $X$, $0_X$ is the origin of $X$, and $0_{X^*}$ is the origin of $X^*$. Note that the dual space $X^*$ of the real normed space $X$ is known to be a real Banach space.

For any set $\Omega \subseteq X$ we denote by $\cl\,\Omega$, $\cl_w\,\Omega$, $\intt\,\Omega$, $\cor\,\Omega$, $\bd\,\Omega$ and $\mbox{conv}\,\Omega$ the closure, weak closure, interior, algebraic interior, boundary, and  convex hull, respectively. 
For any $a,b \in X$ we define 
\begin{align*}
    [a,b] := \{(1-t) a + t b\mid t \in [0,1]\},\\
    (a,b) := \{(1-t) a + t b\mid t \in (0,1)\},\\
    (a,b] := \{(1-t) a + t b\mid t \in (0,1]\}.
\end{align*}

The following auxiliary lemma will be used to prove Lemma \ref{lemma_relative_position}.

\begin{lemma} \label{lema_segmentos}
Consider two sets $A,B \subseteq X$ and assume that 
$\intt\, A \neq \emptyset$. Suppose that there exists $a\in \intt\,A$ such that for every $b \in B$ we have
\begin{equation}\label{condicion_frontera}
(a,b] \cap \bd\,A = \emptyset.
\end{equation}
Then, $B \subseteq \intt\,A$.
\end{lemma}

\begin{proof}
Suppose, contrary to our claim, that $B \not\subseteq A$, and pick some $b \in B\setminus A$. Consider the continuous function $t \mapsto f(t) := (1-t)a + tb$ and define the non-empty and bounded set $I:=\{t\in [0,1] \mid f(t)\in A\}$. Denote $\alpha:=\sup I \in [0,1]$. It is clear that $f(\alpha)\in \cl A$. Note that the case $\alpha = 0$ can not appear, since $f(0) \in \intt\, A$ and $f$ is continuous (thus $f(\varepsilon) \in A$ for some $\varepsilon > 0$).
Consider two cases:

\textit{Case 1:} Let $\alpha=1$. Then, $b=f(\alpha)\in (\cl A) \setminus A$, hence $b \in (a,b]\cap \bd\, A$. 

\textit{Case 2:} Let $\alpha \in (0,1)$. We can pick two sequences $(t_n)_n\subseteq I$ and $(s_n)_n\subseteq [0,1]\setminus I$ that converge to $\alpha$, i.e., $\lim_{n \to \infty} t_n = \alpha = \lim_{n \to \infty} s_n$. On the one hand, 
$\lim_{n \to \infty} f(t_n) = f(\alpha)  \in (a,b) \cap \cl A$, and on the other hand, 
since $f(s_n) \in X \setminus A$ for all $n \in \mathbb{N}$, we get $\lim_{n \to \infty} f(s_n) =  f(\alpha)\in X \setminus \intt\,A$. 
Thus, $f(\alpha)\in (a,b)\cap \bd\, A$. 

In both cases, we get a contradiction to \eqref{condicion_frontera}.
\qed
\end{proof}

\subsection{Cones and their properties}

A cone  $C\subseteq X$ (i.e., a set $C\subseteq X$ with  $0_X \in C = \mathbb{R}_+ \cdot C$) is said to be non-trivial if $\{0_X\} \neq C \neq X$; pointed if $\ell(C) = \{ 0_X\}$, where $\ell(C) := C \cap (-C)$ is the lineality of $C$; acute if $\cl\,C$ is pointed; solid if ${\rm int}\,C \neq \emptyset$; convex if $C$ is a convex set (or equivalently, if $C+C = C$).   

\begin{definition} \label{def:top_base_cone}
Consider a cone $C \subseteq X$. A set $B \subseteq C$ is called a {base} for $C$, if
	$B$ is a non-empty set,  and $C = \mathbb{R}_+ \cdot B$ with $0_X \notin \cl\, B$.
    Moreover, $C$ is said to be well-based if there exists a bounded, convex base of $C$.
\end{definition}

\begin{remark} \label{rem:bases}
Assume that $B$ is a base in the sense of Definition \ref{def:top_base_cone} for the cone $C \subseteq X$. Then,
$\emptyset \neq C \setminus \{0_X\} = \mathbb{P} \cdot B$, 
and if $B$ is convex, then $C$ is non-trivial and convex.
\end{remark}

\begin{remark} \label{rem:bases_2}
Thanks to the Hahn-Banach theorem, a well-based cone \( C \) can be defined as one for which there exists a bounded convex subset \( B \) such that \( 0_X \not\in \cl\,B \) and, for every \( x \in C \) with \( x \neq 0_X \), there exist unique \( \lambda > 0 \) and \( b \in B \) such that \( x = \lambda b \) (this definition is used, for instance, in \cite{GarciaCastano2015}).
\end{remark}

\begin{definition} \label{def:normlikebase}
Consider a cone $C \subseteq X$. The set
$$
B_C := \{x \in C \mid ||x|| = 1\} = C \cap \mathbb{S}_X
$$
is called the \textbf{norm-base} of $C$.
\end{definition}

\begin{remark}
    Assume that $C \neq \{0_X\}$. Then, $B_{C}$ is a base (in the sense of the Definition \ref{def:top_base_cone}) for $C$, which is, as a subset of $\mathbb{B}_X$, bounded. Note, however, that $B_{C}$ is not necessarily a convex base. 
\end{remark}
    
In what follows, given the norm-base $B_C$ of a cone $C \subseteq X$, we will use the following two convex enlarged sets of $B_C$, 
$$
S_C := \conv(B_C) \quad \text{and} \quad S_C^0 := \conv(\{0_X\} \cup B_C).
$$

Some important relationships between the cones $\conv\, C$ and $\cl(\conv\, C)$ and the sets $S_C$, $S_C^0$, $\cl\, S_C$ and $\cl\, S_C^0$ are discussed in the next lemma.

\begin{lemma} \label{lem:clconvCandclSC}
    Assume that $C \subseteq X$ is a cone with $C \neq \{0_X\}$. Then, the following assertions are valid:
    \begin{itemize}
        \item[$1^\circ$] $\conv\, C = \mathbb{R}_+ \cdot S_C = \mathbb{R}_+ \cdot S_C^0$ and if further $0_X \notin \cl\, S_C$ or $C$ is convex, then $\cl(\conv\, C) = \mathbb{R}_+ \cdot \cl\, S_C = \mathbb{R}_+ \cdot \cl\, S_C^0$.
        \item[$2^\circ$] If $0_X \notin \cl\, S_C$, then $\cl\, S_C$ (respectively, $S_C$) is a bounded convex base for $\cl(\conv\, C)$ (respectively, for $\conv\, C$), and $\cl(\conv\, C) \setminus \{0_X\} = \mathbb{P} \cdot \cl\, S_C$.
    \end{itemize}
\end{lemma}
\begin{proof} The result for the case $C = X$ is obvious. From \cite[Rem. 2.5]{gunther2023nonlinear_reflexive} we get $\conv\, C = \mathbb{R}_+ \cdot S_C$, and if further $0_X \notin \cl\, S_C$, then $\cl(\conv\, C) = \mathbb{R}_+ \cdot \cl\, S_C$.
Moreover, the first part of assertion $2^\circ$ is discussed in \cite[Rem. 2.5]{gunther2023nonlinear_reflexive}, while the second part $\cl(\conv\, C) \setminus \{0_X\} = \mathbb{P} \cdot \cl\, S_C$ follows from $0_X \notin \cl\, S_C$ and $\cl(\conv\, C)  = \mathbb{R}_+ \cdot \cl\, S_C$ (provided by $1^\circ$). 

Let us prove the remaining statements in $1^\circ$. 
Of course, $\conv\, C = \mathbb{R}_+ \cdot S_C \subseteq \mathbb{R}_+ \cdot S_C^0 \subseteq \conv\, C$. 
Consider two cases:

\textit{Case 1:} Let $0_X \notin \cl\, S_C$. Then,
$\cl(\conv\, C) = \mathbb{R}_+ \cdot \cl\, S_C \subseteq \mathbb{R}_+ \cdot \cl\, S_C^0 \subseteq \cl(\conv\, C)$, which means that $\cl(\conv\, C) = \mathbb{R}_+ \cdot \cl\,S_C = \mathbb{R}_+ \cdot \cl\,S_C^0$.

\textit{Case 2:} Let $0_X \in \cl\, S_C$ and $C$ be convex. By  \cite[Th. 3.1 (2)]{gunther2023nonlinear_reflexive} we have $\cl\,S_{{\rm cl}\, C}^0 = ({\rm cl}\, C) \cap \mathbb{B}_X$, while Lemma \ref{lem:no_cl_vs_no_cl} ($2^\circ$) ensures $\cl\,S_{{\rm cl}\, C}^0 = \cl\,S_{C}^0$. Consequently, we get 
${\rm  cl}\, C = \mathbb{R}_+ \cdot B_{{\rm  cl}\, C} \subseteq \mathbb{R}_+ \cdot (({\rm cl}\, C) \cap \mathbb{B}_X) = \mathbb{R}_+ \cdot \cl\,S_{{\rm cl}\, C}^0 = \mathbb{R}_+ \cdot \cl\,S_{C}^0 \subseteq \cl\,C,$
which means that $\cl\,C = \mathbb{R}_+ \cdot \cl\,S_C^0$. The remaining equality $\cl\,C = \mathbb{R}_+ \cdot \cl\,S_C$ follows from the fact that $0_X \in \cl\,S_C = \cl\,S_C^0$ (by \cite[Th. 3.1 (2)]{gunther2023nonlinear_reflexive}).

\qed
\end{proof}

Given a cone $C\subseteq X$, its dual cone is defined by 
$$C^+:=\{x^* \in X^*\mid \forall x \in C:\;x^*(x)\geq 0\}.$$ 
Furthermore, the subset 
$$C^{\#}:=\{x^* \in X^*\mid \forall x \in C \setminus \{0_X\}:\; x^*(x) > 0\}$$
is of interest. Obviously, both sets $C^+$ and $C^\#$ are convex for any (not necessarily convex) cone $C \subseteq X$, and if $C^\# \neq \emptyset$, then $C$ is pointed. Moreover, one has 
\begin{align}
  C^{+} & = (\conv\,C)^+ = (\cl(\conv\,C))^+,\notag\\
  (\cl(\conv\,C))^\#  \subseteq C^{\#} & = (\conv\,C)^\#, 
\label{eq:K+=convK+=clvonvK+}   
\end{align}
but the inclusion $(\cl(\conv\,C))^\# \subseteq C^{\#}$ can be strict (see G\"opfert et al.  \cite[p. 55]{GoeRiaTamZal2023}).

\begin{remark} \label{rem:clconvKwell-based} 
    Consider a non-trivial cone $C \subseteq X$ in a real normed space $X$.
    Taking into account \eqref{eq:K+=convK+=clvonvK+}, it is well-known (see \cite[Th. 3.6]{Chiang2012}, \cite[Prop. 2.2.23 and 2.2.32]{GoeRiaTamZal2023}, \cite[Sec. 2.2, Th. 3.1]{gunther2023nonlinear_reflexive}) that
    \begin{equation} \label{eq:KbasedKSharpneqEmpty}
    \conv\,C \text{ has a convex base} \iff  C^\# \neq \emptyset,\end{equation} 
    and  
    \begin{align} \label{eq:KwellbasedIntK+neqEmpty}
        \conv\,C \text{ is well-based} & \iff \cl(\conv\,C) \text{ is well-based}\notag\\
        & \iff  \intt\,C^{+}\neq\emptyset \iff \cor\,C^{+}\neq\emptyset  \\
        & \iff  0_X \notin \cl\, S_{\cl(\conv\,C)} \iff  0_X \notin \cl\, S_C \notag.
    \end{align} 
    In particular, for every $x^* \in C^{\#}$ the set $B :=\{x \in \conv\,C \mid x^*(x) = 1\}$ is a convex base for $\conv\,C$. 
    According to \cite[Th. 1.1]{GarciaCastano2015} (see also \cite{GarciaCastano2019}), we have further  
    \begin{align} 
        \conv\,C \text{ is well-based}&\iff  0_X \notin \cl(\conv((\conv\,C) \setminus \intt\, \mathbb{B}_X)) \notag\\
        & \;\;\qquad (\text{i.e., } 0_X \text{ is a denting point for } \conv\,C)
        \label{eq:denting_point}\\
        &\iff  \exists\, x^* \in X^* \colon\; 0<\inf _{x \in S_{\conv \,C}} x^*(x). \notag
    \end{align} 
    As an easy consequence of \eqref{eq:KbasedKSharpneqEmpty} and \eqref{eq:KwellbasedIntK+neqEmpty} we also get
    \begin{align*} 
        \conv\,C \text{ is well-based} & \;\Longrightarrow\;  \cl(\conv\,C) \text{ has a convex base}\notag\\
        & \iff (\cl(\conv\,C))^\# \neq \emptyset \label{eq:well-basedC_clC_pointed}\\
        & \;\Longrightarrow\;  \cl(\conv\,C) \text{ is pointed (i.e., } \conv\,C \text{ is acute}).\notag
    \end{align*}
    In addition, if $X$ has finite dimension and $C$ is closed, then $C^\# = \intt\,C^{+}$ (see \cite[Th. 2.1(4), Rem. 2.6]{gunther2023nonlinear_reflexive}), hence all the conditions involved in \eqref{eq:KbasedKSharpneqEmpty}, \eqref{eq:KwellbasedIntK+neqEmpty} and  \eqref{eq:denting_point} are equivalent.
\end{remark}

\subsection{Separation of convex sets / cones}

Let us present some results related to separation of convex sets / cones. The first proposition will be a key result for proving some main characterizations of symmetric and non-symmetric separation conditions in the upcoming Sections \ref{sec:sym_and_non-sym_separation} and \ref{sec:sym_and_non-sym_cone_separation}.

\begin{proposition}
    \label{prop:basic separation}
	Consider two non-empty convex sets $\Omega^1, \Omega^2 \subseteq X$. Consider the following assertions:
    \begin{itemize}
    \item[$1^\circ$] $0_X \notin \cl(\Omega^2 - \Omega^1)$.
    \item[$2^\circ$] There is $x^* \in X^* \setminus \{0_{X^*}\}$ such that
    $\sup_{x \in \Omega^1}\, x^*(x) < \inf_{y \in \Omega^2}\, x^*(y)$.
    \item[$3^\circ$]  $\Omega^1 \cap \Omega^2 = \emptyset$.
    \end{itemize}
    Then, $1^\circ \Longleftrightarrow 2^\circ \Longrightarrow 3^\circ$, and if further $\Omega^1$ and $\Omega^2$ are closed and one of these sets is weakly compact, then $3^\circ \Longrightarrow 2^\circ$.
\end{proposition}

\begin{proof} $1^\circ \Longrightarrow 2^\circ$: Assume that $1^\circ$ is valid. By classical linear separation of $0_X$ from the non-empty, closed, convex set $\cl(\Omega^2 - \Omega^1)$, there is $x^* \in X^* \setminus \{0_X\}$ such that
$$0 = x^*(0_X) < \gamma := \inf_{z \in \cl(\Omega^2 - \Omega^1)}\, x^*(z) \leq x^*(y) - x^*(x)$$
for all $y \in \Omega^2$ and $x \in \Omega^1$. In particular, noting that the latter formula yields $\sup_{x \in \Omega^1}\,x^*(x) < \infty$, we derive
$$0 < \gamma \leq x^*(y) - \sup_{x \in \Omega^1}\,x^*(x),$$
hence
$$
\sup_{x \in \Omega^1}\,x^*(x) <  \gamma  +  \sup_{x \in \Omega^1}\,x^*(x) \leq x^*(y) \quad \text{for all } y \in \Omega^2.
$$
Thus, $\sup_{x \in \Omega^1}\,x^*(x) <  \gamma  +  \sup_{x \in \Omega^1}\,x^*(x) \leq \inf_{y \in \Omega^2}\, x^*(y)$, which shows $2^\circ$ is true.

\par

$2^\circ \Longrightarrow 1^\circ$: 
Assume that $2^\circ$ is valid. Then, for all $y \in \Omega^2$ and $x \in \Omega^1$, we have
$$
0 <  \gamma := \inf_{y \in \Omega^2}\, x^*(y) - \sup_{x \in \Omega^1}\,x^*(x) \leq x^*(y) - x^*(x) = x^*(y -  x),
$$
hence, for all $z \in \cl(\Omega^2 - \Omega^1)$,
$$
0 < \gamma \leq x^*(z). 
$$
Of course, this shows that $0_X \notin \cl(\Omega^2 - \Omega^1)$.

\par 
Of course, $2^\circ \Longrightarrow 3^\circ$ is valid. If $\Omega^1$ and $\Omega^2$ are closed (hence weakly closed) and one of these sets is weakly compact, then the classical strict linear separation result for convex sets yields $3^\circ \Longrightarrow 2^\circ$.
    \qed
\end{proof}

\begin{remark} \label{rmk: equivalent_condition} 
Note, for any non-empty sets $\Omega^1, \Omega^2 \subseteq X$, we have
$$
0_X \notin \cl(\cl\,\Omega^2 - \cl\,\Omega^1) \iff 0_X \notin \cl(\Omega^2 - \Omega^1)  \iff d(\Omega^1, \Omega^2) > 0,
$$
where $d(\Omega^1, \Omega^2)$ is the classical distance between the sets $\Omega^1$ and $\Omega^2$, which is defined by 
$$
d(\Omega^1,\Omega^2) := \inf\{||x - y|| \mid x \in \Omega^1, y \in \Omega^2\}.
$$
\end{remark}

The following proposition recalls a known linear separation result for closed, convex cones (see \cite[Prop. 2.2]{gunther2023nonlinear_reflexive} and also \cite[Th. 3.22]{Jahn2011}).

\begin{proposition}[{\cite[Prop. 2.2]{gunther2023nonlinear_reflexive}}] \label{prop:linear_cone_separation}
	Suppose that $C, K \subseteq X$ are non-trivial, closed, convex cones and $\cl\,S_C$ is weakly compact with $0_X \notin \cl\,S_C$. If $C \cap K = \{0_X\}$, then there is $x^* \in X^* \setminus \{0_{X^*}\}$ such that 
    \begin{equation}
        \label{eq:sep_lin_cones}
        x^*(k) \geq 0 > x^*(c) \quad \text{for all } k \in K \mbox{ and  }c \in C \setminus \{0_X\}.
    \end{equation}
\end{proposition} 

The classical weak linear separation for convex cones (in the sense of Eidelheit) is formulated in the next proposition.

\begin{proposition}
    \label{prop:linear_cone_separation_weak}
	Suppose that $C, K \subseteq X$ are convex cones, and $\intt\,C \neq \emptyset$. 
    If $(\intt\,C) \cap K = \emptyset$, then there is $x^* \in X^* \setminus \{0_{X^*}\}$ such that 
    \begin{equation*}
        \label{eq:sep_lin_cones_weak}
        x^*(k) \geq 0 > x^*(c) \quad \text{for all } k \in K \mbox{ and  }c \in \intt\,C.
    \end{equation*}
\end{proposition} 

\subsection{Augmented dual cones and Bishop-Phelps cones} \label{sec:ADCandBPC}

Given a cone $C \subseteq X$ with $C \neq \{0_X\}$ and the norm-base $B_C$, the so-called augmented dual cone of $C$, introduced by Kasimbeyli (Gasimov) in \cite{Gasimov}, \cite{Kasimbeyli2010}, is defined by
$$
C^{a+}  :=\{(x^*,\alpha)\in C^{+}\times \RR_+ \mid  \forall x \in C:\; x^*(x)-\alpha\| x \| \geq 0\}.
$$
Moreover, we consider the following subsets of $C^{a+}$,
\begin{align*}
    C^{a\#} & :=\{(x^*,\alpha)\in C^{\#}\times \RR_+ \mid \forall x \in C \setminus \{0_X\}:\;  x^*(x)-\alpha\| x \| > 0\},\\
    C^{aw\#} & :=\{(x^*,\alpha)\in C^{\#}\times \RR_+ \mid \forall x \in \cl_w B_C:\;  x^*(x) > \alpha\}.
\end{align*}
Clearly, we have $C^{aw\#}\subseteq C^{a\#} \subseteq C^{a+}$ and
\begin{align*}
    C^{a+} & =\{(x^*,\alpha)\in C^{+}\times \RR_+ \mid C \subseteq C(x^*, \alpha)\} \\
    & = \{(x^*,\alpha)\in C^{+}\times \RR_+ \mid \forall x \in B_C:\;  x^*(x) \geq \alpha\},\\
    C^{a\#} & =\{(x^*,\alpha)\in C^{\#}\times \RR_+ \mid C \setminus \{0_X\} \subseteq C^>(x^*, \alpha)\} \\
    & = \{(x^*,\alpha)\in C^{\#}\times \RR_+ \mid \forall x \in B_C:\;  x^*(x) > \alpha\}.
\end{align*}
In addition, we have 
\begin{align}
    \label{eq:clconvCa+=Ca+}
    C^{a+} =(\cl C)^{a+}& = (\conv\,C)^{a+} = (\cl(\conv\,C))^{a+},\notag\\
    (\cl(\conv\,C))^{a\#}  \subseteq C^{a\#}& = (\conv\,C)^{a\#}.
\end{align} 
If $(x^*,\alpha) \in C^{a\#} \cap (X^* \times \mathbb{P})$ and $||x^*||_* > \alpha > 0$, then $C(x^*, \alpha) \in \mathcal{C}_{{\rm BP}}$.

\medskip
In the following proposition, we will examine some properties of the augmented dual cone $C^{a+}$ involving its subsets $C^{aw\#}$, $C^{a\#}$, and ${\rm cor}\,C^{a+}$.

\begin{proposition}[{cf. \cite{gunther2023nonlinear_reflexive}, \cite{GarciaCastano2015,GarciaCastano2024}}] \label{prop:properties_augmented_dual_cone}
	For any cone $C\subseteq X$ with $C \neq \{0_X\}$ the following assertions hold:
    \begin{itemize}
    \item[$1^\circ$] $C^{aw\#} \cap (X^* \times \mathbb{P})\neq \emptyset  \iff C^{a\#} \cap (X^* \times \mathbb{P})\neq \emptyset \iff C^{a+} \cap (X^* \times \mathbb{P}) \neq \emptyset \\
    \iff 0_X \notin {\rm cl}\,S_C$.
    \item[$2^\circ$] 
        $\cor\, C^{a+} = \{(x^*, \alpha) \in C^+ \times \mathbb{P} \mid \inf_{x  \in B_C}\,  x^*(x) > \alpha \} \subseteq C^{aw\#} \cap (X^*\times \mathbb{P})$.
    \item[$3^\circ$] 
        If $\cl_w\,B_C$ is weakly compact, then
        $\cor\, C^{a+} = C^{aw\#} \cap (X^*\times \mathbb{P})$, and \\
        $C^{aw\#} \neq \emptyset \iff \cor\, C^{a+} \neq \emptyset \iff 0_X \notin \cl\,S_C$.
    \item[$4^\circ$] 
        If $B_{C}$ is weakly compact, then 
        $\cor\, C^{a+} =  C^{a\#} \cap (X^*\times \mathbb{P})$.	
    \end{itemize}
\end{proposition}

\begin{proof} 
    If $\{0_X\} \neq \cl(\conv\,C) = X$, 
    then $C^{+} = \{0_{X^*}\}$, $C^\# = \emptyset$, $\cor\, C^{+} = \intt\, C^{+} = \emptyset$, $C^{a+} = \{(0_{X^*}, 0)\}$, $\cor\, C^{a+} = C^{a\#} = C^{aw\#} = \emptyset$ and $0_X \in \cl\,S_C$, hence $1^\circ-4^\circ$ are valid. 
    
    Now, assume that $C$ is non-trivial.    
    Then, the assertion $1^\circ$ follows from \cite[Th. 3.1]{gunther2023nonlinear_reflexive} (see also \cite[Th. 1.1]{GarciaCastano2015}, \cite[Lem.3.7]{GarciaCastano2024}), while the assertion $2^\circ$ follows from \cite[Th. 3.2]{gunther2023nonlinear_reflexive}.
    Note that in assertion $2^\circ$ of Proposition \ref{prop:properties_augmented_dual_cone} we write an upper bound $C^{aw\#} \cap (X^*\times \mathbb{P})$ for $\cor\, C^{a+}$ instead of $C^{a\#} \cap (X^*\times \mathbb{P})$ (which was stated in \cite[Th. 3.2]{gunther2023nonlinear_reflexive}). This is possible because the second condition in $2^\circ$ follows directly from the first condition (equality) in $2^\circ$ and the fact that $\inf_{x \in B_C} x^*(x) = \inf_{x \in \cl_w\,B_C} x^*(x)$. \qed
\end{proof}

\begin{lemma} \label{lem:no_cl_vs_no_cl}
For any cone $C\subseteq X$ the following assertions hold:
\begin{itemize}
    \item[$1^\circ$] 
    $\cl\, S_{\cl\,C} = \cl\,S_C$.
    \item[$2^\circ$] 
    $\cl\, S_{\cl\,C}^0 = \cl\,S_C^0$.
\end{itemize}
\end{lemma}

\begin{proof} The case that $C$ is closed (e.g. $C = \{0_X\}$ or $C = X$) is obvious. Now, assume that $C$ is non-trivial.

$1^\circ$. Let us first prove that $B_{\cl\,C} \subseteq \cl\,B_{C}$. 
Indeed, since $B_{\cl\, C}=B_C \cup B_{\bd \,C}$, 
and $B_C \subseteq \cl\,B_{C}$, we just need to prove that $B_{\bd \,C} \subseteq \cl\,B_{C}$.  Take $x_0 \in B_{\bd \, C}=\mathbb{S}_X \cap \bd \, C$. Then, $x_0=\lim_{n \to \infty} x_n$ for some sequence $(x_n)_n\subseteq C \setminus \{0_X\}$. 
Note that $\lim_{n \to \infty} \|x_n\| = \|x_0\| = 1$. Define $u_n\ :=\frac{x_n}{\|x_n\|}\in \mathbb{S}_X \cap C  = B_{C}$ for all $n\in\NN$. Then, $\lim_{n \to \infty} u_n = \frac{x_0}{\|x_0\|} = x_0$, and consequently $x_0 \in \cl\,B_{C}$.  
Therefore, $B_{\cl \,C} \subseteq \cl\,B_{C} \subseteq \cl(\conv\,B_{C}) = \cl\,S_{C}$. By the convexity of $\cl\,S_{C}$ we get 
$S_{\cl\,C} = \conv\,B_{\cl \,C} \subseteq \cl\,S_C$. 
Moreover, from $S_{\cl\,C} \subseteq \cl\,S_C$ we also get $\cl\, S_{\cl\,C} \subseteq \cl\,S_C$. Noting that $\cl\, S_{\cl\,C} \supseteq \cl\,S_C$ is obvious, we conclude $\cl\, S_{\cl\,C} = \cl\,S_C$.

$2^\circ$. Taking into account that $B_{\cl \,C} \subseteq \cl\,B_{C}$, we get 
$B_{\cl \,C} \cup \{0_X\} \subseteq (\cl\,B_{C}) \cup \{0_X\} = \cl(B_{C} \cup \{0_X\}) \subseteq \cl(\conv(B_{C} \cup \{0_X\})) = \cl \,S_C^0$. Then, by the convexity of $\cl \,S_C^0$, we infer 
$S_{\cl\,C}^0 = \conv(B_{\cl \,C} \cup \, \{0_X\}) \subseteq \cl\, S_C^0$. 
Finally, $\cl\, S_{\cl\,C}^0 = \cl\,S_C^0$ follows immediately.
\qed
\end{proof}

In the next lemma we state some more properties of the closed cones $C(x^*,\alpha)$ (with $(x^*, \alpha) \in X^*\times \mathbb{R}$) that we introduced in Section \ref{sec:introduction}.

\begin{lemma} \label{lem:properties_BP}
    Consider any $(x^*, \alpha) \in (X^*\setminus \{0_{X^*}\}) \times \mathbb{R}$. Then, we have:
    \begin{itemize}
        \item[$1^\circ$] The following equalities are valid:
        \begin{align*}
    C(x^*,\alpha) & = -S(x^*,\alpha) = S(-x^*,\alpha) = \{x \in X \mid \varphi_{-x^*, \alpha}(x) \leq 0\},\\
    C^>(x^*,\alpha) & = -S^<(x^*,\alpha)  = S^<(-x^*,\alpha) = \{x \in X \mid \varphi_{-x^*, \alpha}(x) < 0\},\\
    X \setminus C(x^*,\alpha) & = C^>(-x^*,-\alpha) = S^<(x^*,-\alpha) = \{x \in X \mid \varphi_{-x^*, \alpha}(x) > 0\},\\
    X \setminus C^>(x^*,\alpha) & = C(-x^*,-\alpha) = S(x^*,-\alpha) = \{x \in X \mid \varphi_{-x^*, \alpha}(x) \geq 0\}.
\end{align*}
\item[$2^\circ$] $C(x^*,\alpha) \neq X$ if and only if  $\alpha > -||x^*||_*$.
\item[$3^\circ$] If $\alpha \in (-||x^*||_*, ||x^*||_*)$, then $C(x^*,\alpha)$ is non-trivial,  the sets
\begin{align*}
    \intt\,C(x^*,\alpha) & = C^>(x^*,\alpha) = \{x \in X \mid \varphi_{-x^*, \alpha}(x) < 0 \},\\
    \bd\,C(x^*,\alpha) & = C(x^*,\alpha) \setminus C^>(x^*,\alpha) = \{x \in X \mid \varphi_{-x^*, \alpha}(x) = 0 \},\\
    (\bd\,C(x^*,\alpha)) \setminus \{0_X\} & = \{x \in X\setminus \{0_X\} \mid \varphi_{-x^*, \alpha}(x) = 0 \}
\end{align*}
are non-empty, and
\begin{equation}
    \label{eq:sup_alpha}
    \sup_{x \in S_{\bd \, C(x^*, \alpha)}} \xi_\alpha \cdot x^*(x) = \sup_{x \in S_{\bd \, C(x^*, \alpha)}^0} \xi_\alpha \cdot  x^*(x) = |\alpha|,
\end{equation}
where $\xi_\alpha := 1$ if $\alpha \geq 0$, and $\xi_\alpha := -1$ if $\alpha < 0$.
\end{itemize}

\end{lemma}

\begin{proof}
    Suppose that $x^* \neq 0_{X^*}$ (hence $X \neq \{0_X\}$). 
    The proof of the equalities given in $1^\circ$ is obvious, while the equivalence given in $2^\circ$ follows immediately from the equivalence of the following statements:
    \begin{itemize}
        \item[$\bullet$] $X = C(x^*,\alpha)$.
        \item[$\bullet$] $X = \ell(C(x^*,\alpha)) = \{x\in X \mid -|x^*(x)| \geq \alpha ||x||\}$.
        \item[$\bullet$] 
        $||x^*||_* = \sup_{x \neq 0_X}\, \left|x^*\left(\frac{x}{||x||}\right)\right| \leq -\alpha$.
    \end{itemize}

    It remains to prove $3^\circ$. The conclusion is well-known for the case $\alpha \in [0, ||x^*||_*)$ (see also Remark \ref{rem:BPcone}). Now, assume that $\alpha \in (-||x^*||_*, 0)$. By Remark \ref{rem:BPcone} we get that $C(-x^*,-\alpha)$ is a non-trivial, closed, solid, convex cone with $C^>(-x^*,-\alpha) = \intt\,C(-x^*,-\alpha) \neq \emptyset$ and $(\bd\,C(-x^*,-\alpha))\setminus \{0_X\} \neq \emptyset$. Hence,
     \begin{align*}
        \bd\,C(x^*,\alpha) & = 
        \bd\,\{x \in X \mid \varphi_{-x^*, \alpha}(x) \leq 0\} \\
        & = \bd\,\{x \in X \mid \varphi_{-x^*, \alpha}(x) > 0\}\\
        & = \bd\,C^>(-x^*,-\alpha)\\
        & = \bd(\intt\,C(-x^*,-\alpha))\\
        & = \bd\,C(-x^*,-\alpha)\\
        & = \{x \in X \mid \varphi_{-x^*, \alpha}(x) = 0\}\\
        & = C(x^*,\alpha) \setminus C^>(x^*,\alpha),
    \end{align*}
    and so
    $$
    \intt\,C(x^*,\alpha) = \{x \in X \mid \varphi_{-x^*, \alpha}(x) < 0\} = C^>(x^*,\alpha).
    $$    
    Since $x^* \neq 0_{X^*}$ there is $\bar x \in X$ such that $x^*(\bar x) > 0$. If $\alpha < 0$, then $x^*(\bar x) > 0 > \alpha ||\bar x||$, hence $\bar x \in C^>(x^*,\alpha) = \intt\,C(x^*,\alpha)$, i.e., $C(x^*,\alpha)$ is solid. The latter fact, together with the assertion $2^\circ$, gives the non-triviality of $C(x^*,\alpha)$. Of course, $(\bd\,C(x^*,\alpha)) \setminus \{0_X\} = (\bd\,C(-x^*,-\alpha))\setminus \{0_X\} \neq \emptyset$.
        
    It remains to prove \eqref{eq:sup_alpha}. Take any $\bar x \in S_{\bd \, C(x^*, \alpha)}^0$, i.e., $\bar x = \sum_{i = 1}^l \lambda_i x^i$ for some $x^1, \ldots, x^l \in (B_{\bd \, C(x^*, \alpha)}) \cup \{0_X\}$ and $\lambda_1, \ldots, \lambda_l \geq 0$ with $\sum_{i = 1}^l \lambda_i = 1$. Consider two cases:
    
    \textit{Case 1:} If $\alpha \in [0, ||x^*||_*)$, then
    $x^*(\bar x) = \sum_{i = 1}^l \lambda_i x^*(x^i) \leq  \sum_{i = 1}^l \lambda_i \alpha = \alpha,$
    taking into account that $x^*(0_X) = 0$ and $x^*(x^i) = \alpha$ if $x^i \in (B_{\bd \, C(x^*, \alpha)}) \setminus \{0_X\}$. Consequently, for $\tilde{x} \in B_{\bd \, C(x^*, \alpha)} \; (\neq \emptyset)$ we have
    $$
    \alpha = x^*(\tilde{x}) \leq \sup_{x \in S_{\bd \, C(x^*, \alpha)}} x^*(x) \leq \sup_{x \in S_{\bd \, C(x^*, \alpha)}^0} x^*(x) \leq \alpha.
    $$
        
    \textit{Case 2:} If $\alpha \in (-||x^*||_*, 0)$, then
    $x^*(\bar x) = \sum_{i = 1}^l \lambda_i x^*(x^i) \geq  \sum_{i = 1}^l \lambda_i \alpha = \alpha.$
    Consequently, for $\tilde{x} \in B_{\bd \, C(x^*, \alpha)}  \; (\neq \emptyset)$
    we have
    $$
    \alpha = x^*(\tilde{x}) \geq \inf_{x \in S_{\bd \, C(x^*, \alpha)}} x^*(x) \geq \inf_{x \in S_{\bd \, C(x^*, \alpha)}^0} x^*(x) \geq \alpha.
    $$

    We conclude the validity of \eqref{eq:sup_alpha}.
    \qed
\end{proof}

We now establish a technical result, which will be needed to prove Proposition \ref{prop:basic separation_BP_bd}.

\begin{lemma}\label{lemma_relative_position}
Consider a cone $C\subseteq X$ and take some $x^*\in X^*$ and $\alpha\in (-\|x^*\|_*,\|x^*\|_*)$. Then, the following assertions are valid:
\begin{itemize}
    \item[$1^\circ$] If $\alpha < \inf_{x \in S_{\bd \,C}} x^*(x)$, then exactly one of the following inclusions holds: $\intt C \subseteq C(x^*,\alpha)$
    or  $X \setminus\intt C \subseteq C(x^*,\alpha)$.
    \item[$2^\circ$]  If $\sup_{x \in S_{\bd \,C}} x^*(x) < \alpha$, then exactly one of the following inclusions holds: $C(x^*,\alpha) \subseteq X\setminus \intt C$   or $ C(x^*,\alpha) \setminus\{0_X\}\subseteq \intt C$.
\end{itemize}
\end{lemma}

\begin{proof} 
$1^\circ$. In view of Lemma \ref{lem:properties_BP} ($2^\circ$), under our assumption $\alpha\in (-\|x^*\|_*,\|x^*\|_*)$, we have $C(x^*,\alpha) \neq X$, so that both inclusions cannot be valid at the same time. It is clear that exactly one of the following conditions holds: $\intt C \subseteq C(x^*,\alpha)$ or $\intt C \not \subseteq C(x^*,\alpha)$. 
The case $\intt C = \emptyset$ is clear. Suppose $\intt C \neq \emptyset$. Of course, $\intt C \subseteq C(x^*,\alpha)$ can happen and we are done. Assume that $\intt C \not\subseteq C(x^*,\alpha)$, and denote $\Omega := X\setminus C(x^*,\alpha)$. Note that $\Omega = C^>(-x^*,-\alpha) = \intt\,C(-x^*,-\alpha)$ by Lemma \ref{lem:properties_BP} ($3^\circ$). Fix some $x_0\in (\intt C) \cap \Omega$.  We will prove that $\Omega\subseteq \intt C$, which is equivalent to $X \setminus\intt C \subseteq C(x^*,\alpha)$. Let us pick an arbitrary $y \in \Omega$. 
By the definition of $\Omega$ and the inequality $\alpha<\inf_{x \in S_{\bd \,C}} x^*(x)$ we get $\alpha \|x\| \leq x^*(x)$ for $x \in \bd \,C$, and $x^*(x) < \alpha \|x\|$ for $x \in \Omega$, which shows that $\Omega \cap \bd C=\emptyset$. 
Consider two cases:

\textit{Case 1:} Assume that $\alpha\in (-\|x^*\|_*,0]$. In this case,  $\Omega = C^>(-x^*,-\alpha)$ is convex assuring that $[x_0,y] \subseteq \Omega$, hence $[x_0,y]\cap \bd C \subseteq \Omega \cap \bd C =\emptyset$. Then, Lemma \ref{lema_segmentos} applied to $A=C$ and $B=[x_0,y]$ ensures that $y \in \intt C$. 

\textit{Case 2:} Assume that $\alpha \in (0,\|x^*\|_*)$. 
Since in this case, $\emptyset \neq -C^>(x^*,\alpha) = \intt S(x^*,\alpha)$, we can consider $z\in (x_0+\intt S(x^*,\alpha))\cap (y+\intt S(x^*,\alpha))$ (otherwise, assuming that the latter set is empty, one gets a contradiction by using a classical linear separation argument for convex sets, taking into account that $S(x^*,\alpha)$ is a solid, convex cone). 
Now, we claim that $[x_0,z]\subseteq \Omega$. Indeed, take any $x \in [x_0,z]$ and write $x = x_0 + \lambda (z-x_0) $ with $\lambda \in [0,1]$. Then,
\begin{align*}
     x^*(x) - \alpha \|x\| & =  x^*(x_0 + \lambda (z-x_0)) - \alpha\|x_0 + \lambda (z-x_0)\| \\
     &\leq  x^*(x_0)+x^*(\lambda (z-x_0))-\alpha (\|x_0\|-\lambda \|z-x_0\| )\\
     & =  x^*(x_0)-\alpha \|x_0\| + \lambda (x^*(z-x_0) + \alpha \|z-x_0 \|)\\
     & <0,
\end{align*}
where the latter strict inequality is ensured by $x_0 \in \Omega$ and $z-x_0 \in \intt S(x^*,\alpha)$. This shows that $x \in C^>(-x^*,-\alpha) = \Omega$.
As a consequence, we get $[x_0,z]\cap \bd C \subseteq \Omega \cap \bd C = \emptyset$, and by Lemma \ref{lema_segmentos} applied to $A = C$ and $B = [x_0,z]$ we conclude $z\in \intt C$. 
We also have $[z,y] \subseteq \Omega$. Indeed, for $x = y + \lambda (z-y)$ with $\lambda \in [0,1]$, $z\in (\intt C) \cap \Omega$, $y\in \Omega$, one can analogously prove that $x^*(x) - \alpha \|x\| < 0$, hence $x \in \Omega$. Finally, Lemma \ref{lema_segmentos} applied to $A = C$ and $B = [z,y]$ gives $y \in \intt C$. 

$2^\circ$. Since $\sup_{x \in S_{\bd \,C}} x^*(x) < \alpha$ is equivalent to $-\alpha<\inf_{x \in S_{\bd \,C}} (-x^*)(x)$, by $1^\circ$ we get $\intt C \subseteq C(-x^*,-\alpha)$ (or equivalently, $\intt C \subseteq \intt C(-x^*,-\alpha) = C^>(-x^*,-\alpha) = X\setminus C(x^*,\alpha)$) or $X \setminus\intt C \subseteq C(-x^*,-\alpha)$. 
Equivalently, we also get $C(x^*,\alpha) \subseteq X\setminus \intt C$ or $\intt C(x^*,\alpha) = C^>(x^*,\alpha) = X\setminus  C(-x^*,-\alpha) \subseteq \intt C$. 
From the assumptions in $2^\circ$ we infer $(\bd C(x^*,\alpha)) \cap (\bd C) = \{0_X\}$, hence $\intt C(x^*,\alpha) \subseteq \intt C$ if and only if $C(x^*,\alpha) \setminus\{0_X\}\subseteq \intt C$.
\qed
\end{proof}

\begin{remark} \label{rem:C_BP_cone}
In the previous Lemma \ref{lemma_relative_position}, under the assumptions in $1^\circ$ we have $(\bd C) \setminus \{0_X\} \subseteq C^>(x^*,\alpha) \subseteq C(x^*,\alpha)$, while under the assumptions in $2^\circ$ we have $(\bd C) \setminus \{0_X\} \subseteq C^>(-x^*,-\alpha) = X\setminus C(x^*,\alpha)$. Thus, the assertions in Lemma \ref{lemma_relative_position} can be stated as:
\begin{itemize}
    \item[$1^\circ$] If $\alpha<\inf_{x \in S_{\bd \,C}} x^*(x)$, then exactly one of the following inclusions holds: $\cl\,C \subseteq C(x^*,\alpha)$
    or  $X \setminus C(x^*,\alpha) \subseteq \intt C$.
    \item[$2^\circ$]  If $\sup_{x \in S_{\bd \,C}} x^*(x) < \alpha$, then exactly one of the following inclusions holds: $(\cl\,C)\setminus\{0_X\}\subseteq X\setminus C(x^*,\alpha)$ or $C(x^*,\alpha) \setminus\{0_X\}\subseteq \intt C$.
\end{itemize}  
\end{remark}

\subsection{Basics in non-symmetrical cone separation based on Bishop-Phelps separating cones} \label{sec:basics_nonsym_cone_separation}

Consider a non-trivial cone $C \subseteq X$ with the norm-base $B_{C}$, and a non-trivial cone $K \subseteq X$ with the norm-base $B_{K}$.

\begin{remark} \label{rem:sep_conditions}
For any $(x^*, \alpha) \in C^{a\#} \cap (X^* \times \mathbb{P})$, it is easy to check that the  non-symmetric strict separation condition 
\begin{equation*} \label{separation_condition_NCK}
C(x^*, \alpha) \in \mathcal{N}(C, K)
\end{equation*}
is equivalent to any of the following conditions:
\begin{align}  
&K \setminus \{0_X\} \subseteq X \setminus C(x^*, \alpha) \quad \text{and} \quad  C \setminus \{0_X\} \subseteq \intt\, C(x^*, \alpha) \label{separation_condition_cones_1}; \\
&K \cap C(x^*, \alpha) = \{0_X\} \quad \text{and} \quad  C \setminus \{0_X\} \subseteq \intt\, C(x^*, \alpha)\label{separation_condition_cones_2};\\
    &(-x^*)(k) + \alpha ||k|| >  0 > (-x^*)(c) + \alpha ||c||  \mbox{ for } k \in K\setminus \{0_X\}, c \in C \setminus \{0_X\}\label{separation_condition_inequalities2};\\
    &x^*(k) + \alpha ||k|| >  0 > x^*(c) + \alpha ||c||    \mbox{ for } k \in -K\setminus \{0_X\}, c \in -C \setminus \{0_X\}\label{separation_condition_inequalities}.
\end{align}

Note that conditions \eqref{separation_condition_inequalities2} and \eqref{separation_condition_inequalities} correspond to analytical formulations, while \eqref{separation_condition_cones_1} and \eqref{separation_condition_cones_2} can be viewed as geometrical formulations.

\end{remark}

In the next propositions, we recall known nonlinear cone separation results for (not necessarily convex) cones involving Bishop-Phelps separating cones / separating (norm-linear) functions.

\begin{proposition}[{\cite[Th. 5.2]{gunther2023nonlinear_reflexive}}] \label{th:sep_cones_Banach_3}
Assume that $\cl\,S_{C}$ is weakly compact (e.g. if $X$ is also reflexive). Then, the following conditions are equivalent:
	\begin{itemize}
		\item[$1^\circ$] 
		$
		(\cl\,S_K^0) \cap (\cl\, S_{C}) = \emptyset.
		$
		\item[$2^\circ$]  There exists $(x^*, \alpha) \in C^{aw\#} \cap (X^* \times \mathbb{P})$  such that 
        \eqref{separation_condition_inequalities}
        is valid.
		\item[$3^\circ$]  There exists $(x^*, \alpha) \in C^{aw\#}$ such that  \eqref{separation_condition_inequalities} is valid.
	\end{itemize}	
\end{proposition}

\begin{proposition}[{\cite[Lem. 5.2]{gunther2023nonlinear_reflexive}}] \label{lem:new_lemma_strict_sep_conditions}
The following assertions are equivalent:
   \begin{itemize}
		\item[$1^\circ$]  There exists $(x^*, \alpha) \in \cor\,C^{a+}$ such that \eqref{separation_condition_inequalities} is valid.
        \item[$2^\circ$]  There exists $(x^*, \alpha) \in \cor\,C^{a+}$ such that \eqref{separation_condition_inequalities} with $\cl({\rm conv}\,C)$ in the role of $C$ and $\cl\, K$ in the role of $K$ is valid.
        \item[$3^\circ$]  There exist $\delta_2 > \delta_1 > 0$ and $x^* \in X^*$ such that, for any $\alpha \in (\delta_1, \delta_2)$, we have $(x^*, \alpha) \in \cor\,C^{a+}$, and \eqref{separation_condition_inequalities} with $\cl({\rm conv}\,C)$ in the role of $C$ and $\cl\, K$ in the role of $K$ is valid.  
        \item[$4^\circ$] There exist $\delta_2 > \delta_1 > 0$ and $x^* \in X^*$ such that, for any $\alpha \in (\delta_1, \delta_2)$, we have $(x^*, \alpha) \in C^{a\#}$, and \eqref{separation_condition_inequalities} with $\cl({\rm conv}\,C)$ in the role of $C$ and $\cl\, K$ in the role of $K$ is valid. 
   \end{itemize}  
\end{proposition}

\begin{proposition}[{\cite[Cor. 5.1]{gunther2023nonlinear_reflexive}}]
    Assume that $\cl\, S_{\cl({\rm conv}\,C)}$ is weakly compact. Then, the following assertions are equivalent:  
    \begin{itemize}
        \item[$1^\circ$] $(\cl\,S_K^0) \cap (\cl\, S_{C}) = \emptyset$.
        \item[$2^\circ$] $(\cl\,S_{\cl\, K}^0) \cap (\cl\, S_{\cl({\rm conv}\,C)}) = \emptyset$.
        \item[$3^\circ$] There exists $(x^*, \alpha) \in C^{aw\#} \cap (X^* \times \mathbb{P})$ such that \eqref{separation_condition_inequalities} is valid.
        \item[$4^\circ$] There exists $(x^*, \alpha) \in C^{aw\#} \cap (X^* \times \mathbb{P})$ such that \eqref{separation_condition_inequalities}         
        with $\cl({\rm conv}\,C)$ in the role of $C$ and $\cl\, K$ in the role of $K$ is valid.  
    \end{itemize}
\end{proposition}

\begin{proposition}[{\cite[Th. 5.3]{gunther2023nonlinear_reflexive}}] \label{th:separationBanachSpacePureConvex} Assume that $K$ is  closed and convex, and that $\cl\, S_{C}$ is weakly compact. Then, the following conditions are equivalent:
\begin{itemize}
    \item[$1^\circ$] $(\cl\,S_K^0) \cap (\cl\, S_{C}) = \emptyset$.		
    \item[$2^\circ$] $K \cap \cl({\rm conv}\,C) = \{0_X\}$ and $0_X \notin \cl\, S_{C}$.
\end{itemize}
If the set $\cl\,S_{\cl({\rm conv}\, C)}$ is weakly compact, then any of the assertions $1^\circ$ and $2^\circ$ is equivalent to
\begin{itemize}
    \item[$3^\circ$] $0_X \notin \cl\, S_{C}$ and there exists $x^* \in X^* \setminus \{0_{X^*}\}$ such that \eqref{eq:sep_lin_cones} with $\cl({\rm conv}\,C)$ in the role of $C$ is valid.
\end{itemize}
\end{proposition}

\begin{proposition}[{\cite[Th. 3.1]{GarciaCastano2023}}] \label{lem:GarciaCastano2023}
The following assertions are equivalent:
   \begin{itemize}
		\item[$1^\circ$] $0_X \notin \cl(S_{C} - S_{\bd\,K}^0)$.  
        \item[$2^\circ$] There exist $\delta_2 > \delta_1 > 0$ and $x^* \in X^*$ such that, for any $\alpha \in (\delta_1, \delta_2)$, we have $(x^*, \alpha) \in C^{a\#}$, and \eqref{separation_condition_inequalities} with $\cl({\rm conv}\,C)$ in the role of $C$ and $\bd\,K$ in the role of $K$ is valid.
   \end{itemize}  
\end{proposition}

\section{On Symmetric and Non-Symmetric Separation Conditions Involving Cones} \label{sec:sym_and_non-sym_separation} 

Given two non-trivial cones $C\subseteq X$ and $K \subseteq X$ (with the norm-bases $B_{C}$ and $B_K$) in the real normed space ($X$, $||\cdot||$), we are going to study relationships between symmetrical and non-symmetrical separation conditions that involve the sets $\cl\,S_C^0$, $\cl\,S_K^0$, $\cl\,S_C$ and $\cl\,S_K$, respectively.

\subsection{Non-Symmetric Separation}

First, we present our general result involving our main non-symmetric separation conditions.

\begin{proposition}
    \label{prop:basic separation_BP_a}
    Consider the following assertions:
    \begin{itemize}
    \item[$1^\circ$] $0_X \notin \cl(S_{C} - S_K^0)$.
    \item[$2^\circ$] There exists $x^* \in X^* \setminus \{0_{X^*}\}$ such that $\sup_{k \in \cl\,S_K^0} x^*(k) < \inf_{c \in \cl\, S_{C}} x^*(c)$.
    \item[$3^\circ$] $(\cl\,S_{K}^0) \cap (\cl\, S_{C}) = \emptyset$.
    \item[$4^\circ$] $(\cl\,K) \cap (\cl\, S_{C}) = \emptyset$.
    \item[$5^\circ$] $(\cl\,K) \cap \cl({\rm conv}\,C) = \{0_X\}$  and $0_X \notin \cl\, S_{C}$.
    \end{itemize}
    Then, $1^\circ \Longleftrightarrow 2^\circ \Longrightarrow 3^\circ \Longrightarrow 4^\circ \Longleftrightarrow 5^\circ$.
    If $\cl\,S_K^0$ or $\cl\,S_{C}$ is weakly compact (e.g. if $X$ is also reflexive), then $3^\circ \Longrightarrow 1^\circ$. Moreover, if $K$ is convex, then $5^\circ \Longrightarrow 3^\circ$.
\end{proposition}

\begin{proof} 
    The equivalence $1^\circ \iff 2^\circ$, the implication $1^\circ \Longrightarrow 3^\circ$, and, under the weak compactness assumption, the implication $3^\circ \Longrightarrow 1^\circ$ follow from Proposition \ref{prop:basic separation}. 
    Note that Lemma \ref{lem:no_cl_vs_no_cl} shows that the assertion $3^\circ$ is nothing else than $(\cl\,S_{\cl\,K}^0) \cap (\cl\, S_{\cl\,C}) = \emptyset$. Then, the implications
    $3^\circ \Longrightarrow 4^\circ \Longleftrightarrow 5^\circ$ and, under the convexity of $K$, the implication $5^\circ \Longrightarrow 3^\circ$ follow from \cite[Sec. 4]{gunther2023nonlinear_reflexive}.
    \qed
\end{proof}

The next result shows how the non-symmetric separation condition given in Proposition \ref{prop:basic separation_BP_a} ($1^\circ$) involving the non-trivial cones $C$ and $K$ is related to corresponding separation conditions involving the boundaries / closures of $C$ and $K$.

\begin{proposition}
    \label{prop:basic separation_BP_bd}
    The following assertions are equivalent:
    \begin{itemize}
    \item[$1^\circ$] $0_X \notin \cl(S_{C} - S_{K}^0)$. 
     \item[$2^\circ$] $0_X \notin \cl(S_{\cl\,C} - S_{\cl\,K}^0)$.
    \item[$3^\circ$] $0_X \notin \cl(S_{\bd\,C} - S_{\bd\,K}^0)$ and  $K\cap C = \{0_X\}$.
    \item[$4^\circ$] $0_X \notin \cl(S_{\bd\,C} - S_{\cl\,K}^0)$ and  $K\cap C = \{0_X\}$.
    \item[$5^\circ$] $0_X \notin \cl(S_{\cl\,C} - S_{\bd\,K}^0)$ and  $K\cap C = \{0_X\}$.
    \end{itemize}
\end{proposition}
  
\begin{proof}
Let us first prove $1^\circ \Longleftrightarrow 2^\circ$. The implication $2^\circ \Longrightarrow 1^\circ$ is obvious.
By Remark \ref{rmk: equivalent_condition}, $0_X \notin \cl(S_C - S_K^0) \Longleftrightarrow 0_X \notin \cl(\cl\,S_C - \cl\,S_K^0)$, and from Lemma~ \ref{lem:no_cl_vs_no_cl} we have  $ S_{\cl\,C} \subseteq \cl\,S_C$ and $S_{\cl\,K}^0 \subseteq \cl\,S_K^0$, hence  $0_X \notin \cl(\cl\,S_C - \cl\,S_K^0) \Longrightarrow  0_X \notin \cl(S_{\cl\,C} - S_{\cl\,K}^0)$. This shows that $1^\circ \Longrightarrow 2^\circ$.

We are going to prove $2^\circ \Longrightarrow 3^\circ$. Suppose that $2^\circ$ is valid. Since $S_{\bd\,C} \subseteq S_{\cl\,C}$ and $S_{\bd\,K}^0 \subseteq S_{\cl\,K}^0$, we have $\cl(S_{\bd\,C} - S_{\bd\,K}^0) \subseteq \cl(S_{\cl\,C} - S_{\cl\,K}^0)$, hence $0_X \notin \cl(S_{\cl\,C} - S_{\cl\,K}^0) \Longrightarrow 0_X \notin \cl(S_{\bd\,C} - S_{\bd\,K}^0)$.
By Proposition~ \ref{prop:basic separation} ($1^\circ \Longrightarrow 3^\circ$), we infer that 
$(\cl\,S_{\cl\,C}) \cap (\cl\,S_{\cl\,K}^0) = \emptyset$,
which implies $(\cl\,K) \cap (\cl({\rm conv}\,C)) = \{0_X\}$ (see \cite[Sec. 4]{gunther2023nonlinear_reflexive}). In particular, this shows that $K \cap C = \{0_X\}$.
We conclude that $3^\circ$ is valid.

Let us prove that $3^\circ \Longrightarrow 1^\circ$. Suppose that $3^\circ$ is valid, i.e., $0_X \notin \cl(S_{\bd C} - S_{\bd K}^0)$ and $K \cap C = \{0_X\}$. By Proposition \ref{prop:basic separation} combined with  $0_X \in S_{\bd K}^0$, there exists $x^* \in X^*$ such that  $0\leq \sup_{x \in S_{\bd K}^0} x^*(x) < \inf_{y \in S_{\bd C}} x^*(y)$. 
Furthermore, we have  $\sup_{x \in S_{\bd K}} x^*(x) \leq \sup_{x \in S_{\bd K}^0} x^*(x)$.  
Denote $L_K := \sup_{x \in S_{\bd K}} x^*(x)$, $L_K^0 := \sup_{x \in S_{\bd K}^0} x^*(x)$, $I_C := \inf_{y \in S_{\bd C}} x^*(y)$, and fix $\alpha_1, \alpha_2 \in \mathbb{R}$ such that $L_K \leq L_K^0 < \alpha_1 < \alpha_2 < I_C$. Note that 
$$0 < \alpha_1 < \alpha_2 < I_C \leq\sup_{y \in S_{\bd C}} x^*(y) \leq \sup_{y \in \mathbb{B}_X} |x^*(y)| = ||x^*||_*.$$

On the one hand, by Lemma \ref{lemma_relative_position} ($1^\circ$) and Remark \ref{rem:C_BP_cone}, we have either $C \subseteq C(x^*, \alpha_2)$ or $X \setminus C(x^*, \alpha_2) \subseteq \intt C$. However, the latter inclusion is false; otherwise (taking into account that $L_K < \alpha_2$ and Lemma \ref{lem:properties_BP} ($1^\circ$)) we get $$(\bd K) \setminus \{0_X\} \subseteq C^>(-x^*, -\alpha_2) = X \setminus C(x^*, \alpha_2) \subseteq \intt C.$$ Thus, for every $x \in \bd K$ with $x \neq 0_X$, there exists $y \in K \cap \intt C$, which contradicts the assumption that $K \cap C = \{0_X\}$.  

On the other hand, by Lemma \ref{lemma_relative_position} ($2^\circ$)  and Remark \ref{rem:C_BP_cone}, we have either $K \setminus \{0_X\} \subseteq X \setminus C(x^*, \alpha_1)$ or $C(x^*, \alpha_1) \setminus \{0_X\} \subseteq \intt K$. Also here the latter inclusion is false; otherwise (taking into account that $\alpha_1 < I_C$) we have  
$$(\bd C) \setminus \{0_X\} \subseteq C^>(x^*, \alpha_1) \subseteq C(x^*, \alpha_1) \setminus \{0_X\} \subseteq \intt K.$$
Thus, for every $x \in \bd C$ with $x \neq 0_X$, there exists $y \in (\intt K) \cap C$, which contradicts the assumption that $K \cap C = \{0_X\}$.
 
Therefore, we conclude that $C \subseteq C(x^*, \alpha_2)$ and $K \setminus \{0_X\}\subseteq X \setminus C(x^*, \alpha_1)$. 
Consequently, $B_C \subseteq \{x \in X \mid x^*(x) \geq \alpha_2\}$ and $B_K \subseteq \{x \in X \mid x^*(x) < \alpha_1\}$, hence $S_C - S_K \subseteq \{x \in X \mid x^*(x) \geq \alpha_2\} - \{x \in X \mid x^*(x) < \alpha_1\}$.  
Now, since $0_X \in \{x \in X \mid x^*(x) < \alpha_1\}$, it follows that  $S_C - S_K^0 \subseteq \{x \in X \mid x^*(x) \geq \alpha_2\} - \{x \in X \mid x^*(x) < \alpha_1\}$.  
As $\alpha_1$ and $\alpha_2$ were chosen freely under the condition $L_K^0 <\alpha_1 < \alpha_2 < I_C$, we can strengthen this to  $S_C - S_K^0 \subseteq \{x \in X \mid x^*(x) \geq I_C\} - \{x \in X \mid x^*(x) \leq L_K^0\}$.  Finally, since $0_X \notin \cl(\{x \in X \mid x^*(x) \geq I_C\} - \{x \in X \mid x^*(x) \leq L_K^0\})$ taking into account that $I_C > L^0_K$, we conclude that $0_X \notin \cl(S_C - S_K^0)$, i.e., $1^\circ$ is valid. 

It remains to prove the equivalence of the assertions $3^\circ$, $4^\circ$ and $5^\circ$. Under the assumption $K \cap C = \{0_X\}$, we have 
\begin{align*}
    0_X \notin \cl(S_{\cl\,C} - S_{\cl\,K}^0) & \Longrightarrow 0_X \notin \cl(S_{\bd\,C} - S_{\cl\,K}^0) \\
    & \Longrightarrow 0_X \notin \cl(S_{\bd\,C} - S_{\bd\,K}^0)\\
    & \Longrightarrow 0_X \notin \cl(S_{\cl\,C} - S_{\cl\,K}^0),
\end{align*}
which shows that $3^\circ\Longleftrightarrow 4^\circ$,
and
\begin{align*}
    0_X \notin \cl(S_{\cl\,C} - S_{\cl\,K}^0) & \Longrightarrow 0_X \notin \cl(S_{\cl\,C} - S_{\bd\,K}^0)\\
    & \Longrightarrow 0_X \notin \cl(S_{\bd\,C} - S_{\bd\,K}^0)\\
    & \Longrightarrow 0_X \notin \cl(S_{\cl\,C} - S_{\cl\,K}^0),
\end{align*}
which shows that $3^\circ\Longleftrightarrow 5^\circ$.
\qed
\end{proof}

\begin{remark}\label{remark_both_conditions}
The condition ``$K \cap C = \{0_X\}$'' in Proposition~\ref{prop:basic separation_BP_bd} can be replaced by the condition ``$(\cl K) \cap (\cl({\rm conv}\,C)) = \{0_X\}$'' or by the condition ``$(\cl K) \cap (\cl C) = \{0_X\}$''. The proofs are analogous. Note that Propositions \ref{prop:basic separation_BP_a} and \ref{prop:basic separation_BP_bd} provide interesting insights into the separation conditions discussed in the non-symmetric separation approach in the works 
\cite{GarciaCastano2023,GarciaCastano2024}, \cite{gunther2023nonlinear,gunther2023nonlinear_reflexive} and \cite{Kasimbeyli2010}.
\end{remark}

\subsection{Symmetric strict separation}

The results in the following proposition can be used to connect the symmetric cone separation approach with the non-symmetric cone separation approach.

\begin{proposition}
    \label{prop:basic separation_BP_e}
    The following assertions are equivalent:
    \begin{itemize}
    \item[$1^\circ$] $0_X \notin \cl(S_{C} - S_K)$.
    \item[$2^\circ$] There exists $x^* \in X^* \setminus \{0_{X^*}\}$ such that $\sup_{k \in \cl\,S_{K}} x^*(k) < \inf_{c \in \cl\, S_{C}} x^*(c)$.
    \item[$3^\circ$] There exists $x^* \in X^* \setminus \{0_{X^*}\}$ such that \\$\sup_{k \in \cl\,S_{K}^0} x^*(k) < \inf_{c \in \cl\, S_{C}} x^*(c)$ or $\sup_{k \in \cl\,S_{K}} x^*(k) < \inf_{c \in \cl\, S_{C}^0} x^*(c)$.
    \item[$4^\circ$] $0_X \notin \cl(S_{C} - S_K^0)$ or $0_X \notin \cl(S_{C}^0 - S_K)$.
    \end{itemize}
\end{proposition}

\begin{proof}
    The equivalence of $1^\circ$ and $2^\circ$, respectively, of $3^\circ$ and $4^\circ$ follows from Proposition \ref{prop:basic separation}. 
    Taking into account that 
    $\sup_{k \in \cl\,S_{K}} x^*(k) \leq \sup_{k \in \cl\,S_{K}^0} x^*(k)$ and
    $\inf_{c \in \cl\, S_{C}^0} x^*(c) \leq \inf_{c \in \cl\, S_{C}} x^*(c)$, the implication $3^\circ \Longrightarrow 2^\circ$ is clear. Let us prove the reverse implication. Assume that $2^\circ$ is valid, i.e., there is $x^* \in X^* \setminus \{0_{X^*}\}$ such that $\sup_{k \in \cl\,S_{K}} x^*(k) < \inf_{c \in \cl\, S_{C}} x^*(c)$. Consider cases:
    
    \textit{Case 1:} Assume that $0 \in (\sup_{k \in \cl\,S_{K}} x^*(k), \inf_{c \in \cl\, S_{C}} x^*(c))$. Then, we have $\sup_{k \in \cl\,S_{K}^0} x^*(k) = 0 < \inf_{c \in \cl\, S_{C}} x^*(c)$ and $\sup_{k \in \cl\,S_{K}} x^*(k) < 0 = \inf_{c \in \cl\, S_{C}^0} x^*(c)$.
  
    \textit{Case 2:} Assume that $\sup_{k \in \cl\,S_{K}} x^*(k) \geq 0 = x^*(0_X)$ and pick some $\gamma \in (\sup_{k \in \cl\,S_{K}} x^*(k), \inf_{c \in \cl\, S_{C}} x^*(c))$. Then, $x^*(k) \leq \gamma$ for all $k \in B_{K} \cup \{0_X\}$, hence 
    $$\sup_{k \in \cl\,S_{K}^0} x^*(k) \leq \gamma < \inf_{c \in \cl\, S_{C}} x^*(c).$$

    \textit{Case 3:}  Assume that $\inf_{c \in \cl\, S_{C}} x^*(c) \leq 0 = x^*(0_X)$ and pick some $\gamma \in (\sup_{k \in \cl\,S_{K}} x^*(k), \inf_{c \in \cl\, S_{C}} x^*(c))$. Then, $x^*(c) \geq \gamma$ for all $c \in B_{C} \cup \{0_X\}$, hence 
    $$\inf_{c \in \cl\, S_{C}^0} x^*(c)) \geq \gamma > \sup_{k \in \cl\,S_{K}} x^*(k).$$
    
    Thus, $3^\circ$ is true. 
    
    \qed
\end{proof}

The next proposition relates (under further assumptions, e.g., the weak compactness of $\cl\,S_K$ or $\cl\,S_{C}$) the separation conditions mentioned in Proposition \ref{prop:basic separation_BP_e} to some other classical separation conditions involving cones.

\begin{proposition}
    \label{prop:basic separation_BP_f}
     Consider the following assertions:
    \begin{itemize}
    \item[$1^\circ$] $0_X \notin \cl(S_{C} - S_K)$.
    \item[$2^\circ$] $(\cl\,S_{K}^0) \cap (\cl\, S_{C})  = \emptyset$ or $(\cl\,S_{K}) \cap (\cl\, S_{C}^0)  = \emptyset$.
    \item[$3^\circ$] $(\cl\,S_{K}) \cap (\cl\, S_{C})  = \emptyset$. 
    \item[$4^\circ$] $(\cl K) \cap ({\rm cl}\, S_{C}) = \emptyset$ or $({\rm cl}\, S_{K}) \cap (\cl C) = \emptyset$.
    \item[$5^\circ$] ($(\cl K) \cap \cl({\rm conv}\,C) = \{0_X\}$ and $0_X \notin \cl\, S_{C}$) or ($\cl({\rm conv}\,K) \cap (\cl C) = \{0_X\}$ and $0_X \notin \cl\, S_{K}$).
    \item[$6^\circ$] $(\cl K) \cap  (\cl C) = \{0_X\}$ and [$0_X \notin  \cl\, S_{K}$ or  $0_X \notin \cl\, S_{C}$].
    \end{itemize}
    Then, $1^\circ \Longrightarrow 2^\circ \Longrightarrow  3^\circ  \Longrightarrow 6^\circ$, and moreover, $2^\circ \Longrightarrow  4^\circ  \Longleftrightarrow  5^\circ \Longrightarrow 6^\circ$.\\
    If $\cl\,S_K$ or $\cl\,S_{C}$ is weakly compact, then $3^\circ \Longrightarrow 1^\circ$ (i.e., $1^\circ-3^\circ$ are equivalent). \\
    If $C$ is convex or $K$ is convex, then $3^\circ \Longrightarrow 4^\circ$.\\
    If $C$ and $K$ are convex, then $6^\circ \Longrightarrow 5^\circ \Longrightarrow 2^\circ$  (i.e., $2^\circ-6^\circ$ are equivalent).
\end{proposition}

\begin{proof} 
$1^\circ \Longrightarrow 2^\circ$ follows from Propositions \ref{prop:basic separation_BP_a} and \ref{prop:basic separation_BP_e}, while the implication $2^\circ \Longrightarrow 3^\circ$ and the implication $5^\circ \Longrightarrow  6^\circ$ are obvious. Moreover, $2^\circ  \Longrightarrow 4^\circ \Longleftrightarrow  5^\circ$ is an immediate consequence of Proposition \ref{prop:basic separation_BP_a} ($3^\circ \Longrightarrow 4^\circ \Longleftrightarrow 5^\circ$).

$3^\circ \Longrightarrow 6^\circ$: Assume that $3^\circ$ is valid. It is clear that $0_X \notin  \cl\, S_{K}$ or $0_X \notin \cl\, S_{C}$. On the other hand, by contradiction, suppose that $x\in (\cl K) \cap  (\cl C)$, $x\not = 0_X$. Then, there exists $(x_n)_n\subseteq K$ such that $\lim_{n \to \infty}\frac{x_n}{\|x_n\|}=\frac{x}{\|x\|}\in \cl(K\cap \mathbb{S}_X) = \cl\, B_K \subseteq \cl\,S_K$. Analogously, we can show that $\frac{x}{\|x\|}\in \cl(C\cap \mathbb{S}_X) = \cl\, B_C  \subseteq \cl\,S_C$, which is impossible by $3^\circ$.
We conclude the validity of $6^\circ$.

By Proposition \ref{prop:basic separation}, if $\cl\,S_K$ or $\cl\,S_{C}$ is weakly compact, then $3^\circ \Longrightarrow 1^\circ$. 

W.l.o.g. assume that $C$ is convex.

$3^\circ \Longrightarrow 4^\circ$: Suppose that $3^\circ$ is valid. If $0_X \in \cl\,S_{K}$, then 
$\cl\,S_{K}^0 = \cl\,S_{K}$ by \cite[Th. 3.1 (2)]{gunther2023nonlinear_reflexive}, hence $(\cl K) \cap ({\rm cl}\, S_{C})=\emptyset$ by Proposition \ref{prop:basic separation_BP_a} ($3^\circ \Longrightarrow 4^\circ$). In the previous result one can also change the roles of $C$ and $K$. 
Now, assume that $0_X \notin (\cl\,S_{K}) \cup (\cl\,S_{C})$.
As $3^\circ \Longrightarrow 6^\circ$ is valid, we have $(\cl K) \cap  (\cl C) = \{0_X\}$. Now, by convexity of $C$ we have $\cl S_C\subseteq \cl C$, hence $(\cl S_C) \cap ({\rm cl}\, K) = \emptyset$ (taking into account that $0_X\not \in \cl S_C$). This shows that $4^\circ$ is valid.

Finally, if $C$ and $K$ are convex, then $5^\circ \Longrightarrow 2^\circ$ follows from Proposition \ref{prop:basic separation_BP_a} ($5^\circ \Longrightarrow 3^\circ$), while the implication $6^\circ \Longrightarrow 5^\circ$ is obvious.\qed
\end{proof}

Also for the symmetric case, we relate the separation condition given in Proposition \ref{prop:basic separation_BP_e} ($1^\circ$) involving the non-trivial cones $C$ and $K$ to corresponding separation conditions involving the boundaries / closures of $C$ and $K$.

\begin{proposition}\label{prop:basic separation_BP_bd_1}
    The following assertions are equivalent:
    \begin{itemize}
    \item[$1^\circ$] $0_X \notin \cl(S_C - S_K)$.
    \item[$2^\circ$] $0_X \notin \cl(S_{\cl\,C} - S_{\cl\,K})$.    
    \item[$3^\circ$] $0_X \notin \cl(S_{\bd\,C} - S_{\bd\,K})$ and  $K \cap C = \{0_X\}$.
    \item[$4^\circ$] $0_X \notin \cl(S_{\bd\,C} - S_{\cl\,K})$ and $K \cap C = \{0_X\}$.
    \item[$5^\circ$] $0_X \notin \cl(S_{\cl\,C} - S_{\bd\,K})$ and $K \cap C = \{0_X\}$.
    \end{itemize}
\end{proposition}
\color{black}

\begin{proof} 
Let us first prove $1^\circ \Longleftrightarrow 2^\circ$. The implication $2^\circ \Longrightarrow 1^\circ$ is obvious.
By Remark \ref{rmk: equivalent_condition}, $0_X \notin \cl(S_C - S_K) \Longleftrightarrow 0_X \notin \cl(\cl\,S_C - \cl\,S_K)$, and from Lemma~ \ref{lem:no_cl_vs_no_cl} we have  $ S_{\cl\,C} \subseteq \cl\,S_C $ and $S_{\cl\,K} \subseteq \cl\,S_K $, hence  $0_X \notin \cl(\cl\,S_C - \cl\,S_K) \Longrightarrow  0_X \notin \cl(S_{\cl\,C} - S_{\cl\,K})$. Thus, $1^\circ \Longrightarrow 2^\circ$ is valid.

We also get the equivalence of $2^\circ-5^\circ$, since
 Propositions \ref{prop:basic separation_BP_bd} and \ref{prop:basic separation_BP_e} ensure on the one hand that
    \begin{align*}
       0_X \notin \cl(S_{\cl\,C} - S_{\cl\,K}) 
      & \Longleftrightarrow  \left[ 0_X \notin \cl(S_{\cl\,C} - S_{\cl\,K}^0) \text{ or } 0_X \notin \cl(S_{\cl\,C}^0 - S_{\cl\,K}) \right]\\
      & \Longrightarrow K \cap C = \{0_X\},
     \end{align*}
 and on the other hand, under the condition $K \cap C = \{0_X\}$,
    \begin{align*}
    & \left[0_X \notin \cl(S_{\cl\,C} - S_{\cl\,K}^0) \text{ or } 0_X \notin \cl(S_{\cl\,C}^0 - S_{\cl\,K})\right]\\
       \Longleftrightarrow \; & \left[0_X \notin \cl(S_{\bd\,C} - S_{\bd\,K}^0) \text{ or } 0_X \notin \cl(S_{\bd\,C}^0 - S_{\bd\,K}) \right] \Longleftrightarrow 0_X \notin \cl(S_{\bd\,C} - S_{\bd\,K})\\
     \Longleftrightarrow  \;  &  \left[0_X \notin \cl(S_{\bd\,C} - S_{\cl\,K}^0) \text{ or } 0_X \notin \cl(S_{\bd\,C}^0 - S_{\cl\,K}) \right] \Longleftrightarrow 0_X \notin \cl(S_{\bd\,C} - S_{\cl\,K}) \\
      \Longleftrightarrow  \;  &  \left[0_X \notin \cl(S_{\cl\,C} - S_{\bd\,K}^0) \text{ or } 0_X \notin \cl(S_{\cl\,C}^0 - S_{\bd\,K}) \right]  \Longleftrightarrow  0_X \notin \cl(S_{\cl\,C} - S_{\bd\,K}).  
    \end{align*}
    \qed
\end{proof}

\begin{remark}
The condition ``$K \cap C = \{0_X\}$'' in Proposition~\ref{prop:basic separation_BP_bd_1} can be replaced either by the condition ``$(\cl K) \cap (\cl C) = \{0_X\}$'' or, alternatively, by the condition 
``$(\cl K) \cap (\cl({\rm conv}\,C)) = \{0_X\}$ or $(\cl({\rm conv}\,K)) \cap (\cl C) = \{0_X\}$''. The proofs are analogous.
\end{remark}

\section{On Symmetric and Non-Symmetric Cone Separation Based on Bishop-Phelps Separating Cones}\label{sec:sym_and_non-sym_cone_separation}

The aim of this section is to present new (non-)symmetrical strict cone separation theorems in the real normed space $(X, ||\cdot||)$.  Consider the cones $C$ and $K$ from Section \ref{sec:sym_and_non-sym_separation}.

\subsection{Non-Symmetric Strict Cone Separation}
Let us first discuss the non-symmetrical case. We like to find characterizations for
$\mathcal{N}(C, K \mid \mathcal{C}_{{\rm BP}^*}) \neq \emptyset
$ 
for some family of cones  
$\mathcal{C}_{{\rm BP}^*} \subseteq \mathcal{C}_{{\rm BP}}$, which then also ensure that $C$ and $K$ can be strictly separated by a Bishop-Phelps cone $C(x^*, \alpha)$ (with $||x^*||_* > \alpha > 0$) in a non-symmetric way, i.e., 
$\mathcal{N}(C, K \mid \mathcal{C}_{{\rm BP}}) \neq \emptyset$.
More precisely, for any non-trivial cone $D \subseteq X$, we consider a family $\mathcal{C}_{{\rm BP}^*}$ given by one of the following families
\begin{align*}
    \mathcal{C}_{{\rm BP}^*_1}(D) & := \{C(x^*, \alpha) \mid (x^*, \alpha) \in \cor\,D^{a+}\},\\
    \mathcal{C}_{{\rm BP}^*_2}(D) & := \{C(x^*, \alpha) \mid (x^*, \alpha) \in D^{aw\#} \cap (X^* \times \mathbb{P})\}.
\end{align*}

In view of Proposition \ref{prop:properties_augmented_dual_cone} ($2^\circ$) we have $\mathcal{C}_{{\rm BP}^*_1}(D) \subseteq \mathcal{C}_{{\rm BP}^*_2}(D) \subseteq \mathcal{C}_{{\rm BP}}$, and if $\cl_w B_D$ is weakly compact, then  $\mathcal{C}_{{\rm BP}^*_1}(D) = \mathcal{C}_{{\rm BP}^*_2}(D)$.
In addition, we are interested in a family of cones given by
\begin{align*}
    \mathcal{C}_{{\rm BP}^*_3}(D) & := \{C(x^*, \alpha) \mid (x^*, \alpha) \in D^{aw\#}\}.
\end{align*}
Note that $\mathcal{C}_{{\rm BP}^*_2}(D) \subseteq \mathcal{C}_{{\rm BP}^*_3}(D) \subseteq \mathcal{C}_{{\rm Lin}} \cup \mathcal{C}_{{\rm BP}}$ and $\mathcal{C}_{{\rm BP}^*_3}(D) \setminus \mathcal{C}_{{\rm BP}^*_2}(D) \subseteq \mathcal{C}_{{\rm Lin}}$.

Let us first present our general result (without involving a weak compactness assumption), which in particular gives characterizations for $\mathcal{N}(C, K \mid \mathcal{C}_{{\rm BP}^*_1}(C)) \neq \emptyset$.

\begin{theorem} \label{th:non-sym_sep_theorem_1}
The following conditions are equivalent:
	\begin{itemize}
        \item[$1^\circ$] $0_X \notin \cl(S_{C} - S_K^0)$.
        \item[$2^\circ$] There exists $x^* \in X^* \setminus \{0_{X^*}\}$ such that, for any 
        $$\alpha \in (\sup_{k \in \cl\,S_{K}^0} x^*(k), \inf_{c \in \cl\, S_{C}} x^*(c)) \quad(\subseteq (0,+\infty)),$$
        we have $(x^*, \alpha) \in \cor\,C^{a+}$ and 
        \eqref{separation_condition_inequalities}
        is valid, where $\sup_{k \in \cl\,S_{K}^0} x^*(k) < \inf_{c \in \cl\, S_{C}} x^*(c)$.
        \item[$3^\circ$] There exist $\delta_2 > \delta_1 > 0$ and $x^* \in X^* \setminus \{0_{X^*}\}$ such that, for any $\alpha \in (\delta_1, \delta_2)$, we have $(x^*, \alpha) \in \cor\,C^{a+}$ and \eqref{separation_condition_inequalities} is valid. 
        \item[$4^\circ$] There exists $(x^*, \alpha) \in \cor\,C ^{a+}$ such that \eqref{separation_condition_inequalities} is valid (i.e., $\mathcal{N}(C, K \mid \mathcal{C}_{{\rm BP}^*_1}(C)) \neq \emptyset$).
	\end{itemize}	
    Moreover, the above assertions are also equivalent to the corresponding assertions $1^\circ$--$4^\circ$ where the cone $C$ is replaced by $\cl({\rm conv}\, C)$ (respectively, $\cl\, C$) while $K$ is replaced by $\cl\,K$. 
\end{theorem}

\begin{proof}
By Proposition \ref{prop:basic separation} assertion $1^\circ$ is valid if and only if there is $x^* \in X^* \setminus \{0_{X^*}\}$ such that $\sup_{k \in \cl\,S_{K}^0} x^*(k) < \inf_{c \in \cl\, S_{C}} x^*(c)$.
Since $0_X \in \cl\,S_{K}^0$ we have $0 < \alpha < \inf_{c \in \cl\, S_{C}} x^*(c)$ for any $\alpha \in (\sup_{k \in \cl\,S_{K}^0} x^*(k), \inf_{c \in \cl\, S_{C}} x^*(c))$, which shows that $(x^*, \alpha) \in \cor\, C^{a+}$ by Proposition \ref{prop:properties_augmented_dual_cone} ($2^\circ$). Moreover, it is easy to check that $\sup_{k \in \cl\,S_{K}^0} x^*(k) < \alpha < \inf_{c \in \cl\, S_{C}} x^*(c)$
implies \eqref{separation_condition_inequalities}. 
This shows that $1^\circ \Longleftrightarrow 2^\circ$. The implication $2^\circ \Longrightarrow 3^\circ$ is obvious.
The equivalence $3^\circ \Longleftrightarrow 4^\circ$ is provided by Proposition \ref{lem:new_lemma_strict_sep_conditions}. 
Let us show that the implication $4^\circ \Longrightarrow 1^\circ$ is obvious. Assume that $4^\circ$ is valid, that is, there exists $(x^*, \alpha) \in \cor\,C^{a+}$ such that \eqref{separation_condition_inequalities} is valid. 
Consequently, in view of Proposition \ref{prop:properties_augmented_dual_cone} ($2^\circ$), we have 
$0 < \alpha < \inf_{c \in B_C} x^*(c) = \inf_{c \in \cl\, S_{C}} x^*(c)$. The left inequality in \eqref{separation_condition_inequalities} yields $x^*(k) \leq \alpha$ for all $k \in B_K \cup \{0_X\}$, hence $x^*(k) \leq \alpha$ for all $k \in \cl\, S_{K}^0$, and so $\sup_{k \in \cl\,S_K^0} x^*(k) \leq \alpha$.
We conclude that $\sup_{k \in \cl\,S_{K}^0} x^*(k) \leq \alpha < \inf_{c \in \cl\, S_{C}} x^*(c)$, which means that $1^\circ$ is valid (by Proposition \ref{prop:basic separation}). 

Let us prove the remaining part of the theorem. 
Proposition \ref{lem:new_lemma_strict_sep_conditions} ($1^\circ \Longleftrightarrow 2^\circ$) yields the equivalence of assertion $4^\circ$ and assertion $4^\circ$ with $\cl({\rm conv}\, C)$ in the role of $C$ and $\cl\,K$ in the role of $K$. Moreover, Proposition \ref{prop:basic separation_BP_bd} ($1^\circ \Longleftrightarrow 2^\circ$) yields the equivalence of assertion $1^\circ$ and assertion $1^\circ$ with $\cl\,C$ in the role of $C$ and $\cl\,K$ in the role of $K$. It is easy to check that the assertions $1^\circ-4^\circ$, the assertions $1^\circ-4^\circ$ with $\cl({\rm conv}\, C)$ in the role of $C$ and $\cl\,K$ in the role of $K$, and the assertions $1^\circ-4^\circ$ with $\cl\, C$ in the role of $C$ and $\cl\,K$ in the role of $K$ all imply the conditions $\{0_X\} \subsetneq \cl\, C \subseteq \cl({\rm conv}\, C) \neq X$ and $\{0_X\} \subsetneq \cl\, K \neq X$. 
Then the proven results applied for the non-trivial cones $C$ and $K$, for the non-trivial cones $\cl({\rm conv}\, C)$ and $\cl\,K$, and for the non-trivial cones $\cl\,C$ and $\cl\,K$ yield the result.

\qed
\end{proof}

A similar result (involving the separation property \eqref{separation_condition_inequalities}) we also get in the case that we change the roles of $C$ and $K$. In particular, we derive a characterization for $\mathcal{N}(K, C \mid \mathcal{C}_{{\rm BP}^*_1}(K)) \neq \emptyset$.

\begin{theorem} \label{th:non-sym_sep_theorem_2}
The following conditions are equivalent:
	\begin{itemize}
        \item[$1^\circ$] $0_X \notin \cl(S_{C}^0 - S_{K})$.
        \item[$2^\circ$] There exists $x^* \in X^* \setminus \{0_{X^*}\}$ such that, for any $$\alpha \in (\sup_{c \in \cl\,S_{C}^0} x^*(c), \inf_{k \in \cl\, S_{K}} x^*(k)) \quad (\subseteq (0,\infty)),$$ we have $(x^*, \alpha) \in \cor\,K^{a+}$ and 
        \begin{align}
        x^*(c) + \alpha ||c|| &  >  0 > x^*(k) + \alpha ||k||   \quad \mbox{for } k \in -K\setminus \{0_X\}, c \in -C \setminus \{0_X\},  \label{separation_condition_inequalities_new}
        \end{align}	
        is valid, or equivalently,
        \begin{align}
        (-x^*)(c) - \alpha ||c|| &  <  0 < (-x^*)(k) - \alpha ||k||   \quad \mbox{for } k \in -K\setminus \{0_X\}, c \in -C \setminus \{0_X\},  \label{separation_condition_inequalities_new2}
        \end{align}	
        is valid, where $\sup_{c \in \cl\,S_{C}^0} x^*(c) < \inf_{k \in \cl\, S_{K}} x^*(k)$.
        \item[$3^\circ$] There exists $x^* \in X^* \setminus \{0_{X^*}\}$ such that, for any 
        $$\alpha \in (\sup_{k \in \cl\, S_{K}} x^*(k), \inf_{c\in \cl\,S_{C}^0} x^*(c)) \quad(\subseteq (-\infty, 0)),$$
        we have $(x^*, \alpha) \in -\cor\,K^{a+}$ and 
        \eqref{separation_condition_inequalities} is valid, where $\sup_{k \in \cl\, S_{K}} x^*(k) < \inf_{c\in \cl\,S_{C}^0} x^*(c)$.
        \item[$4^\circ$] There exist $\delta_1 < \delta_2 < 0$ and $x^* \in X^* \setminus \{0_{X^*}\}$ such that, for any $\alpha \in (\delta_1, \delta_2)$, we have $(x^*, \alpha) \in -\cor\,K^{a+}$ and \eqref{separation_condition_inequalities} is valid. 
        \item[$5^\circ$] There exists $(x^*, \alpha) \in -\cor\,K^{a+}$ such that \eqref{separation_condition_inequalities} is valid (i.e., $\mathcal{N}(K, C \mid \mathcal{C}_{{\rm BP}^*_1}(K)) \neq \emptyset$).
	\end{itemize}	
    Moreover, the above assertions $1^\circ$--$5^\circ$ are also equivalent to the corresponding assertions where the cone $K$ is replaced by $\cl({\rm conv}\, K)$ (respectively, $\cl\, K$) while $C$ is replaced by $\cl\,C$. 
\end{theorem}

\begin{proof}
First, observe that $1^\circ \Longleftrightarrow 0_X \notin \cl(S_{K} - S_{C}^0)$. Hence, the equivalences $1^\circ\Longleftrightarrow 2^\circ\Longleftrightarrow 4^\circ \Longleftrightarrow 5^\circ$ follow from Theorem \ref{th:non-sym_sep_theorem_1} (applied for $K$ and $C$ in the role of $C$ and $K$). Note that the following statements are equivalent:
\begin{itemize}
    \item[$\bullet$] There exists $(x^*, \alpha) \in \cor\,K^{a+}$ such that \eqref{separation_condition_inequalities} (with $K$ and $C$ in the role of $C$ and $K$) is valid.
    \item[$\bullet$] There exists $(x^*, \alpha) \in \cor\,K^{a+}$ such that \eqref{separation_condition_inequalities_new} is valid.
    \item[$\bullet$] There exists $(x^*, \alpha) \in \cor\,K^{a+}$ such that \eqref{separation_condition_inequalities_new2} is valid.
    \item[$\bullet$] There exists $(x^*, \alpha) \in -\cor\,K^{a+}$ such that \eqref{separation_condition_inequalities} is valid. 
\end{itemize}

Let us prove the equivalence $2^\circ \iff 3^\circ$. Consider $x^* \in X^* \setminus \{0_{X^*}\}$ and $\alpha \in \mathbb{P}$. Put $y^* := -x^*\; (\in X^* \setminus \{0_{X^*}\})$ and $\beta := -\alpha\; (\in -\mathbb{P})$. Noting that
\begin{itemize}
   \item[$\bullet$]  $(x^*, \alpha)$ satisfies 
 \eqref{separation_condition_inequalities_new2} $\iff$ $(y^*, \beta)$ satisfies 
 \eqref{separation_condition_inequalities} (with ($y^*, \beta$) in the role of ($x^*, \alpha$));
 \item[$\bullet$] $(x^*, \alpha) \in \cor\,K^{a+}$ $\iff$ $(y^*, \beta) \in -\cor\,K^{a+}$;
 \item[$\bullet$] $\alpha \in (\sup_{c \in \cl\,S_{C}^0} x^*(c), \inf_{k \in \cl\, S_{K}} x^*(k))$ \\
 $\iff$ $\beta \in (-\inf_{k \in \cl\, S_{K}} (-y^*)(k), -\sup_{c \in \cl\,S_{C}^0} (-y^*)(c))$ \\
 $\iff$ $\beta \in (\sup_{k \in \cl\, S_{K}} y^*(k), \inf_{c \in \cl\,S_{C}^0} y^*(c))$
\end{itemize}
we conclude that $2^\circ \iff 3^\circ$. 

The remaining part of the theorem with the equivalence of the assertions $1^\circ$--$5^\circ$ with the corresponding assertions where the cone $K$ is replaced by $\cl({\rm conv}\, K)$ (respectively, $\cl\, K$) while $C$ is replaced by $\cl\,C$ follows directly from Theorem \ref{th:non-sym_sep_theorem_1} (applied to $K$ and $C$ in the role of $C$ and $K$).
\qed
\end{proof}

Under the weak compactness of the set $\cl\, S_{C}$ or/and $\cl\, S_K^0$ we can state some further characterizations, in particular, we can characterize the condition $\mathcal{N}(C, K \mid \mathcal{C}_{{\rm BP}^*_i}(C)) \neq \emptyset$ for $i \in \{2,3\}$.

\begin{theorem} \label{th:non-sym_sep_theorem_3}
Consider the following assertions:
	\begin{itemize}
        \item[$1^\circ$] $0_X \notin \cl(S_{C} - S_K^0)$.
		\item[$2^\circ$] $(\cl\,S_K^0) \cap (\cl\, S_{C}) = \emptyset$.
		\item[$3^\circ$] There exists $(x^*, \alpha) \in C^{aw\#} \cap (X^* \times \mathbb{P})$  such that \eqref{separation_condition_inequalities} is valid (i.e., $\mathcal{N}(C, K \mid \mathcal{C}_{{\rm BP}^*_2}(C)) \neq \emptyset$).
        \item[$4^\circ$]  There exists $(x^*, \alpha) \in C^{aw\#}$ such that  \eqref{separation_condition_inequalities} is valid (i.e., $\mathcal{N}(C, K \mid \mathcal{C}_{{\rm BP}^*_3}(C)) \neq \emptyset$).
	\end{itemize}	
    Then, $1^\circ \Longrightarrow 2^\circ$ and $1^\circ \Longrightarrow 3^\circ \Longrightarrow 4^\circ$. 
    If one of the sets $\cl\, S_{C}$ and $\cl\, S_K^0$ is weakly compact, then $2^\circ \Longrightarrow 1^\circ$.
    If $\cl\, S_{C}$ is weakly compact, then $4^\circ \Longrightarrow 1^\circ$.
\end{theorem}

\begin{proof} 
$1^\circ \Longrightarrow 2^\circ$ follows by Proposition \ref{prop:basic separation_BP_a}; $1^\circ \Longrightarrow 3^\circ$ follows by Theorem~\ref{th:non-sym_sep_theorem_1} taking into account $\cor\,C^{a+} \subseteq C^{aw\#} \cap (X^* \times \mathbb{P})$ by Proposition \ref{prop:properties_augmented_dual_cone} ($2^\circ$); $3^\circ \Longrightarrow 4^\circ$  follows by the fact that $C^{aw\#} \cap (X^* \times \mathbb{P}) \subseteq C^{aw\#}$. If one of the sets $\cl\, S_{C}$ and $\cl\, S_K^0$ is weakly compact, then $2^\circ \Longrightarrow 1^\circ$ follows by Proposition \ref{prop:basic separation_BP_a}.
If $\cl\, S_{C}$ is weakly compact, then $4^\circ \Longrightarrow 2^\circ$ by Proposition \ref{th:sep_cones_Banach_3}, hence $4^\circ \Longrightarrow 1^\circ$ (by the previous proven implication).
\qed
\end{proof}

\begin{theorem} \label{th:non-sym_sep_theorem_4}
Consider the following assertions:
	\begin{itemize}
        \item[$1^\circ$] $0_X \notin \cl(S_{C}^0 - S_{K})$.
		\item[$2^\circ$] $(\cl\,S_{C}^0) \cap (\cl\, S_{K}) = \emptyset$.
		\item[$3^\circ$] There exists $(x^*, \alpha) \in (-K^{aw\#}) \cap (X^* \times -\mathbb{P})$ such that \eqref{separation_condition_inequalities} is valid (i.e., $\mathcal{N}(K, C \mid \mathcal{C}_{{\rm BP}^*_2}(K)) \neq \emptyset$).
        \item[$4^\circ$]  There exists $(x^*, \alpha) \in -K^{aw\#}$ such that  \eqref{separation_condition_inequalities} is valid (i.e., $\mathcal{N}(K, C \mid \mathcal{C}_{{\rm BP}^*_3}(K)) \neq \emptyset$).
	\end{itemize}	
    Then, $1^\circ \Longrightarrow 2^\circ$ and $1^\circ \Longrightarrow 3^\circ \Longrightarrow 4^\circ$. 
    If one of the sets $\cl\, S_{K}$ and $\cl\, S_{C}^0$ is weakly compact, then $2^\circ \Longrightarrow 1^\circ$.
    If $\cl\, S_{K}$ is weakly compact, then $4^\circ \Longrightarrow 1^\circ$.
\end{theorem}

\begin{proof} 
Taking into account that $1^\circ \Longleftrightarrow 0_X \notin \cl(S_{K} - S_C^0)$, the result follows immediately from Theorem \ref{th:non-sym_sep_theorem_3} (applied for $K$ and $C$ in the role of $C$ and $K$). Note that the following statements are equivalent:
\begin{itemize}
    \item[$\bullet$] There exists $(x^*, \alpha) \in K^{aw\#} \cap (X^* \times \mathbb{P})$ (respectively, $(x^*, \alpha) \in K^{aw\#}$) such that \eqref{separation_condition_inequalities} (with $K$ and $C$ in the role of $C$ and $K$) is valid.
    \item[$\bullet$] There exists $(x^*, \alpha) \in K^{aw\#} \cap (X^* \times \mathbb{P})$ (respectively, $(x^*, \alpha) \in K^{aw\#}$) such that \eqref{separation_condition_inequalities_new} is valid.
    \item[$\bullet$] There exists $(x^*, \alpha) \in (-K^{aw\#}) \cap (X^* \times -\mathbb{P})$ (respectively, $(x^*, \alpha) \in -K^{aw\#}$) such that \eqref{separation_condition_inequalities} is valid. 
\end{itemize}
\qed
\end{proof}

\begin{remark}
    Note that Theorems \ref{th:non-sym_sep_theorem_1} and \ref{th:non-sym_sep_theorem_3} (respectively, Theorems \ref{th:non-sym_sep_theorem_2} and \ref{th:non-sym_sep_theorem_4}) extend the separation results based on the non-symmetric approach derived in \cite{GarciaCastano2023}, \cite{gunther2023nonlinear_reflexive} and \cite{Kasimbeyli2010}. In particular, Theorems \ref{th:non-sym_sep_theorem_1} and \ref{th:non-sym_sep_theorem_2} provide refinements to the separation results appearing in \cite{GarciaCastano2023} by explicitly fixing the bounds on $\alpha$ via the supremum and infimum, and by proving that the pair $(x^*, \alpha)$ lies in the core of the augmented dual cone. Combining these theorems with Proposition \ref{prop:basic separation_BP_bd}, we observe that, when the cones meet only at the origin, the separation property introduced in \cite{GarciaCastano2023} can be obtained by considering only the boundaries of the cones (using an argument similar to that presented in Remark \ref{remark:separation_frontiers} below). On the other hand, our current separation results improve those of \cite[Th. 4.3]{Kasimbeyli2010}, since we obtain separation characterizations in general normed spaces, whereas \cite[Th. 4.3]{Kasimbeyli2010} restricted the analysis to reflexive spaces and, in some cases, required finite-dimensional spaces to establish equivalences. Moreover, \cite[Th. 4.3]{Kasimbeyli2010} requires closedness assumptions, which is not the case here in our results (similar to \cite{GarciaCastano2023} and \cite{gunther2023nonlinear_reflexive}, as can be seen in the restated results in Section \ref{sec:basics_nonsym_cone_separation}). In \cite{gunther2023nonlinear_reflexive}, characterizations based on $(\cl\,S_{K}^0) \cap (\cl\, S_{C}) = \emptyset$ are established (under weak compactness assumptions on $\cl\, S_{C}$), not on $0_X \notin \cl(S_{K}^0 - S_{C})$.
\end{remark}

The following result shows that two given cones, one of which contains the other, can be interpolated by a Bishop-Phelps cone under suitable assumptions.

\begin{theorem} \label{th:Omega1and2_BP_cone_nonsym}
The following assertions are valid:
\begin{itemize}
\item[$1^\circ$] If $0_X \notin \cl(S_{C} - S_{K}^0)$ (in particular, $K \cap C = \{0_X\}$), then there exists $C_{\rm BP} \in\mathcal{C}_{{\rm BP}^*_1}(C)$ such that
\begin{equation}
\label{eq:Omega1Omega2_cond1}
K \cap C_{\rm BP} = \{0_X\} \quad \text{ and } \quad C \setminus \{0_X\} \subseteq \intt\, C_{\rm BP} \quad \text{ and } \quad C \subseteq C_{\rm BP}.
\end{equation}
\item[$2^\circ$] Define $\widehat{K} := (X \setminus K)\cup \{0_X\}$. If $C \subseteq K$ and $0_X \notin \cl(S_{C} - S_{\widehat{K}}^0)$ (both conditions imply $\widehat{K} \cap C = \{0_X\}$), then there exists $C_{\rm BP} \in\mathcal{C}_{{\rm BP}^*_1}(C)$ such that
\begin{equation}
\label{eq:Omega1Omega2_cond2}
C \setminus \{0_X\} \subseteq \intt\, C_{\rm BP} \subseteq C_{\rm BP} \subseteq K,
\end{equation}
which implies
\begin{equation}
\label{eq:Omega1Omega2_cond3}
\intt\,C \subseteq \intt\,C_{\rm BP} \subseteq \intt\,K \quad \text{ and }\quad
C \subseteq C_{\rm BP} \subseteq K.
\end{equation}
\end{itemize}	
\end{theorem}

\begin{proof}
    $1^\circ$. By Theorem \ref{th:non-sym_sep_theorem_1}, if $0_X \notin \cl(S_{C} - S_{K}^0)$, there exists $(x^*, \alpha) \in \cor\,C^{a+}$ such that \eqref{separation_condition_inequalities}, or equivalently \eqref{separation_condition_cones_2}, is valid, i.e., 
    \begin{equation*}
    K \cap C(x^*, \alpha) = \{0_X\} \quad \text{and} \quad  C \setminus \{0_X\} \subseteq \intt\, C(x^*, \alpha).
    \end{equation*}
    Thus, for $C_{\rm BP} := C(x^*, \alpha) \in\mathcal{C}_{{\rm BP}^*_1}(C)$ the statement in \eqref{eq:Omega1Omega2_cond1} is valid.

    $2^\circ$. Under the given assumptions, the cone $\widehat{K}$ is non-trivial. Again by Theorem \ref{th:non-sym_sep_theorem_1}, if $0_X \notin \cl(S_{C} - S_{\widehat{K}}^0)$, then there exists $(x^*, \alpha) \in \cor\,C^{a+}$ such that \eqref{separation_condition_inequalities} (with $\widehat{K}$ in the role of $K$), or equivalently \eqref{separation_condition_cones_1}, is valid, i.e., 
    \begin{equation*}
    \widehat{K} \setminus \{0_X\} \subseteq X \setminus C(x^*, \alpha) \quad \text{and} \quad  C \setminus \{0_X\} \subseteq \intt\, C(x^*, \alpha).
    \end{equation*}
    Thus, for $C_{\rm BP} := C(x^*, \alpha) \in\mathcal{C}_{{\rm BP}^*_1}(C)$ we have 
    \begin{equation*}
    X \setminus K \subseteq X \setminus C_{\rm BP}\quad \text{and} \quad  C \setminus \{0_X\} \subseteq \intt\, C_{\rm BP},
    \end{equation*}
    or equivalently, \eqref{eq:Omega1Omega2_cond2} is valid.
    Noting that $\intt\,C \subseteq C \setminus \{0_X\}$ (since $C \neq X$), the remaining inclusions given in \eqref{eq:Omega1Omega2_cond3} are obvious.\qed
\end{proof}

\begin{remark}\label{remark:separation_frontiers}
We note that, taking into account Proposition \ref{prop:basic separation_BP_bd}, the conclusions of the previous theorem for the cones  $C$ and $K$ also hold under the assumption of the separation property of their boundaries. More specifically, in $1^\circ$, the condition $0_X \notin \cl(S_{C} - S_{K}^0)$ can be replaced by the conditions $0_X \notin \cl(S_{\bd C} - S_{\bd K}^0)$ and $K \cap C = \{0_X\}$. Likewise, in $2^\circ$, the condition $0_X \notin \cl(S_{C} - S^0_{\widehat{K}})$ can be replaced by
$0_X \notin \cl(S_{\bd C} - S_{\bd K}^0)$, while maintaining the inclusion $C\subseteq K$.
\end{remark}

Under a convexity assumption concerning $K$, we get, in addition to the result in Proposition \ref{prop:basic separation_BP_a}, the following non-symmetric cone separation result (involving linear separation).

\begin{proposition} \label{prop:basic separation_BP_b} 
Assume that $K$ is convex. Consider the following assertions:
\begin{itemize}
\item[$1^\circ$] $0_X \notin \cl(S_{C} - S_{K}^0)$.
\item[$2^\circ$]
$0_X \notin \cl\, S_{C}$ and there exists $x^* \in X^* \setminus \{0_{X^*}\}$ such that 
\begin{equation*}
    x^*(k) \geq 0 > x^*(c) \quad \mbox{for all } k \in \cl\,K \mbox{ and  }c \in \cl(\conv\,C)\setminus \{0_X\}.
\end{equation*}
\item[$3^\circ$]
$0_X \notin \cl\, S_{C}$ and there exists $x^* \in X^* \setminus \{0_{X^*}\}$ such that 
\begin{equation*}
    x^*(k) \geq 0 > x^*(c) \quad \mbox{for all } k \in \cl\,S_{K}^0 \mbox{ and  }c \in \cl\, S_{C}.
\end{equation*}
\item[$4^\circ$] There exists $x^* \in X^* \setminus \{0_{X^*}\}$ such that $\sup_{k \in \cl\,S_{K}^0} x^*(k) = 0 < \inf_{c \in \cl\, S_{C}} x^*(c)$.
\end{itemize}	
Then, $4^\circ \Longrightarrow 1^\circ \Longrightarrow 2^\circ \Longleftrightarrow 3^\circ$, and if $\cl\, S_{C}$ is weakly compact, then $3^\circ \Longrightarrow 4^\circ$.
\end{proposition}

\begin{proof} Note that $1^\circ$ means exactly (in view of Proposition \ref{prop:basic separation_BP_a}) that there is $x^* \in X^* \setminus \{0_{X^*}\}$ such that $\sup_{k \in \cl\,S_{K}^0} x^*(k) < \inf_{c \in \cl\, S_{C}} x^*(c)$.
Thus, the implication $4^\circ \Longrightarrow 1^\circ$ is obvious.

$1^\circ \Longrightarrow 2^\circ$: Assume that $1^\circ$ is valid. Of course, $0_X \notin \cl\,S_{C}$. In view of Theorem \ref{th:non-sym_sep_theorem_1} and Remark \ref{rem:sep_conditions}, there exists $(y^*, \alpha) \in \cor\,C^{a+}$ such that
$(\cl\,K) \cap C_{\rm BP} = \{0_X\}$ and 
$\cl(\conv\,C) \setminus \{0_X\} \subseteq \intt\,C_{\rm BP}$ for $C_{\rm BP} := C(y^*, \alpha)$. By the linear weak separation result in Proposition \ref{prop:linear_cone_separation_weak} (noting that $(\cl\,K) \cap (\intt\,C_{\rm BP}) =  \emptyset$ since $0_X \notin \intt\,C_{\rm BP}$) there is $x^* \in X^* \setminus \{0_{X^*}\}$ such that 
$$x^*(k) \geq 0 > x^*(c) \quad \text{for all } k \in \cl\,K \text{ and } c \in \intt\,C_{\rm BP} \,(\supseteq \cl(\conv\,C) \setminus \{0_X\}),$$
hence $2^\circ$ is valid. 

$2^\circ \Longleftrightarrow 3^\circ$: This equivalence follows immediately from the fact that $\mathbb{P} \cdot \cl\,S_{C} = \cl(\conv\,C) \setminus \{0_X\}$ and $\mathbb{R}_+ \cdot \cl\,S_K^0 = \cl(\conv\,K) = \cl\,K$ (by Lemma \ref{lem:clconvCandclSC}).

$3^\circ \Longrightarrow 4^\circ$: Assume that $3^\circ$ holds true. Then, we easily infer $\sup_{k \in \cl\,S_{K}^0} y^*(k) = 0 < y^*(c)$ for all $c \in \cl\, S_{C}$, where $y^* := -x^* \in X^* \setminus \{0_{X^*}\}$. Under the weak compactness of $\cl\, S_{C}$, we conclude $4^\circ$. \qed
\end{proof}

\begin{remark}
    The result in Proposition \ref{prop:basic separation_BP_b} combined with Proposition \ref{prop:basic separation_BP_a}  extends the cone separation result (for non-trivial closed convex cones) mentioned in Proposition \ref{prop:linear_cone_separation} derived under a weak compactness assumption on $\cl\, S_{C}$. 
\end{remark}
   
\subsection{Symmetric strict cone separation}

Let us now discuss the symmetric case. Our aim is to find characterizations for
$\mathcal{S}(C, K \mid \mathcal{C}_{{\rm BP}^*}) \neq \emptyset
$ 
for some family of cones  
$\mathcal{C}_{{\rm BP}^*} \subseteq \mathcal{C}_{{\rm BP}}$, which then also ensure that $C$ and $K$ can be strictly separated by a Bishop-Phelps cone $C(x^*, \alpha)$ (with $||x^*||_* > \alpha > 0$) in a symmetric way, i.e, $\mathcal{S}(C, K \mid \mathcal{C}_{{\rm BP}}) \neq \emptyset$.
More precisely, for any non-trivial cones $D^1, D^2 \subseteq X$, we consider a family $\mathcal{C}_{{\rm BP}^*}$ given by one of the following two families
\begin{align*}
    \mathcal{C}_{{\rm BP}^*_1}(D^1, D^2) & := \mathcal{C}_{{\rm BP}^*_1}(D^1) \cup \mathcal{C}_{{\rm BP}^*_1}(D^2) \\
    & = \{C(x^*, \alpha) \mid (x^*, \alpha) \in (\cor\,(D^1)^{a+}) \cup(\cor\,(D^2)^{a+})\};\\
    \mathcal{C}_{{\rm BP}^*_2}(D^1, D^2) & := \mathcal{C}_{{\rm BP}^*_2}(D^1) \cup \mathcal{C}_{{\rm BP}^*_2}(D^2) \\
    & = \{C(x^*, \alpha) \mid (x^*, \alpha) \in ((D^1)^{aw\#}) \cup ((D^2)^{aw\#}) \cap (X^* \times \mathbb{P})\};
\end{align*}
In view of Proposition \ref{prop:properties_augmented_dual_cone} ($2^\circ$) we have $\mathcal{C}_{{\rm BP}^*_1}(D^1, D^2) \subseteq \mathcal{C}_{{\rm BP}^*_2}(D^1, D^2) \subseteq \mathcal{C}_{{\rm BP}}$, and if $\cl_w B_{D^1}$ and $\cl_w B_{D^2}$ are weakly compact, then  $\mathcal{C}_{{\rm BP}^*_1}(D^1, D^2) = \mathcal{C}_{{\rm BP}^*_2}(D^1, D^2)$.
In addition, we need the following family of cones
\begin{align*}
    \mathcal{C}_{{\rm BP}^*_3}(D^1, D^2) & := \mathcal{C}_{{\rm BP}^*_3}(D^1) \cup \mathcal{C}_{{\rm BP}^*_3}(D^2) \\
    & = \{C(x^*, \alpha) \mid (x^*, \alpha) \in ((D^1)^{aw\#}) \cup ((D^2)^{aw\#})\}.
\end{align*}
Note that $\mathcal{C}_{{\rm BP}^*_2}(D^1, D^2) \subseteq \mathcal{C}_{{\rm BP}^*_3}(D^1, D^2) \subseteq \mathcal{C}_{{\rm Lin}} \cup \mathcal{C}_{{\rm BP}}$ and $\mathcal{C}_{{\rm BP}^*_3}(D^1, D^2) \setminus \mathcal{C}_{{\rm BP}^*_2}(D^1, D^2) \subseteq \mathcal{C}_{{\rm Lin}}$.

Consider $i \in \{1,2,3\}$. In this section, we are able to give  characterizations for the condition
$$
\mathcal{S}(C, K \mid \mathcal{C}_{{\rm BP}^*_i}(C, K)) = \mathcal{N}(C, K \mid \mathcal{C}_{{\rm BP}^*_i}(C, K))  \cup \mathcal{N}(K, C \mid \mathcal{C}_{{\rm BP}^*_i}(C, K))  \neq \emptyset. 
$$
Since
$\mathcal{C}_{{\rm BP}^*_i}(C, K) = \mathcal{C}_{{\rm BP}^*_i}(C) \cup \mathcal{C}_{{\rm BP}^*_i}(K) = \mathcal{C}_{{\rm BP}^*_i}(K, C)$ 
we have
\begin{align*}
    \mathcal{N}(C, K \mid \mathcal{C}_{{\rm BP}^*_i}(C, K)) = \mathcal{N}(C, K \mid \mathcal{C}_{{\rm BP}^*_i}(C)) \cup \mathcal{N}(C, K \mid \mathcal{C}_{{\rm BP}^*_i}(K)),\\
    \mathcal{N}(K, C \mid \mathcal{C}_{{\rm BP}^*_i}(C, K)) = \mathcal{N}(K, C \mid \mathcal{C}_{{\rm BP}^*_i}(C)) \cup \mathcal{N}(K, C \mid \mathcal{C}_{{\rm BP}^*_i}(K)).
\end{align*}
It is easy to check that $\mathcal{N}(C, K \mid \mathcal{C}_{{\rm BP}^*_i}(K)) = \mathcal{N}(K, C \mid \mathcal{C}_{{\rm BP}^*_i}(C)) = \emptyset$, hence
\begin{align*}
    \mathcal{N}(C, K \mid \mathcal{C}_{{\rm BP}^*_i}(C, K)) & = \mathcal{N}(C, K \mid \mathcal{C}_{{\rm BP}^*_i}(C)),\\
    \mathcal{N}(K, C \mid \mathcal{C}_{{\rm BP}^*_i}(C, K)) & = \mathcal{N}(K, C \mid \mathcal{C}_{{\rm BP}^*_i}(K)),\\
    \mathcal{S}(C, K \mid \mathcal{C}_{{\rm BP}^*_i}(C, K)) & = \mathcal{N}(C, K \mid \mathcal{C}_{{\rm BP}^*_i}(C))  \cup \mathcal{N}(K, C \mid \mathcal{C}_{{\rm BP}^*_i}(K)).
\end{align*}

Let us first present our general result (without involving a weak compactness assumption) that gives characterizations for $\mathcal{S}(C, K \mid \mathcal{C}_{{\rm BP}^*_1}(C, K)) \neq \emptyset$.

\begin{theorem} \label{th:sym_sep_theorem_1a}
The following assertions are equivalent:
    \begin{itemize}
    \item[$1^\circ$] $0_X \notin \cl(S_{C} - S_K)$.
    \item[$2^\circ$]  There exists $x^* \in X^* \setminus \{0_{X^*}\}$ such that, for any $$\alpha \in (\sup_{k \in \cl\,S_{K}} x^*(k), \inf_{c \in \cl\, S_{C}} x^*(c)),$$
    we have \eqref{separation_condition_inequalities} is valid, where $\sup_{k \in \cl\,S_{K}} x^*(k) < \inf_{c \in \cl\, S_{C}} x^*(c)$.    
    \item[$3^\circ$]  There exists $x^* \in X^*$ such that $(x^*, \alpha) \in \cor\,C^{a+}$ and \eqref{separation_condition_inequalities}
    are valid for all 
    $$\alpha \in (\sup_{k \in \cl\,S_{K}^0} x^*(k), \inf_{c \in \cl\, S_{C}} x^*(c)) \quad (\subseteq (0,+\infty)),$$ where $\sup_{k \in \cl\,S_{K}^0} x^*(k) < \inf_{c \in \cl\, S_{C}} x^*(c)$,
    or $(x^*, \alpha) \in -\cor\,K^{a+}$ and  
    \eqref{separation_condition_inequalities}
    are valid for all $$\alpha \in (\sup_{k \in \cl\, S_{K}} x^*(k), \inf_{c \in \cl\,S_{C}^0} x^*(c)) \quad (\subseteq (-\infty, 0)),$$ where $\sup_{k \in \cl\, S_{K}} x^*(k) < \inf_{c \in \cl\,S_{C}^0} x^*(c)$.
    \item[$4^\circ$] There exists $(x^*, \alpha) \in (\cor\,C^{a+}) \cup (-\cor\,K^{a+})$ such that \eqref{separation_condition_inequalities} is valid (i.e., $\mathcal{S}(C, K \mid \mathcal{C}_{{\rm BP}^*_1}(C, K)) \neq \emptyset$).
	\end{itemize}	
 Moreover, the assertions $1^\circ$--$4^\circ$ are also equivalent to the corresponding assertions where the cone $C$ is replaced by $\cl\,C$ while $K$ is replaced by $\cl\,K$.
\end{theorem}

\begin{proof} 
By Proposition \ref{prop:basic separation} assertion $1^\circ$ is valid if and only if there is $x^* \in X^* \setminus \{0_{X^*}\}$ such that $\sup_{k \in \cl\,S_{K}} x^*(k) < \inf_{c \in \cl\, S_{C}} x^*(c)$. Moreover, it is easy to check that $\sup_{k \in \cl\,S_{K}} x^*(k) < \alpha < \inf_{c \in \cl\, S_{C}} x^*(c)$
implies \eqref{separation_condition_inequalities}. This shows that $1^\circ \Longleftrightarrow 2^\circ$. By Proposition \ref{prop:basic separation_BP_e} ($1^\circ \Longleftrightarrow 4^\circ$) the assertion $1^\circ$ is equivalent to $0_X \notin \cl(S_{C} - S_K^0)$ or $0_X \notin \cl(S_{C}^0 - S_K)$. Using Theorem \ref{th:non-sym_sep_theorem_1} ($1^\circ \Longleftrightarrow 2^\circ$) and Theorem \ref{th:non-sym_sep_theorem_2} ($1^\circ \Longleftrightarrow 3^\circ$) the latter statement is equivalent to $3^\circ$. By Theorem \ref{th:non-sym_sep_theorem_1} ($2^\circ \Longleftrightarrow 4^\circ$) and Theorem \ref{th:non-sym_sep_theorem_2} ($3^\circ \Longleftrightarrow 5^\circ$) we obtain $3^\circ \Longleftrightarrow 4^\circ$. 

It remains to show that $1^\circ$--$4^\circ$ are also equivalent to the corresponding assertions where the cone $C$ is replaced by $\cl\,C$ while $K$ is replaced by 
$\cl\,K$. It is easy to check that each of these 8 assertions implies the conditions $\{0_X\} \subsetneq \cl\, C  \neq X$ and $\{0_X\} \subsetneq \cl\, K \neq X$. Now, taking into account  
the equivalence of $1^\circ$ with $0_X \notin \cl(S_{\cl\,C} - S_{\cl\, K})
$ in view of Proposition \ref{prop:basic separation_BP_bd_1} ($1^\circ\Longleftrightarrow 2^\circ$), we get the desired equivalences (using the proven results for the non-trivial cones $C$ and $K$, as well as for the non-trivial cones $\cl\,C$ and $\cl\,K$).
\qed
\end{proof}

Under the weak compactness of $\cl\, S_{C}$ or $\cl\, S_K$, we can state some further characterizations; in particular, we can characterize the condition $\mathcal{S}(C, K \mid \mathcal{C}_{{\rm BP}^*_i}(C, K)) \neq \emptyset$ for $i \in \{2,3\}$.

\begin{theorem} \label{th:sym_sep_theorem_1b}
Consider the following assertions:
	\begin{itemize}
        \item[$1^\circ$] 
        $0_X \notin \cl(S_{C} - S_K)$.
        \item[$2^\circ$] 
		$(\cl\,S_{K}) \cap (\cl\, S_{C})  = \emptyset$. 
		\item[$3^\circ$] There exists $(x^*, \alpha) \in (C^{aw\#} \cup -K^{aw\#}) \cap (X^* \times (\mathbb{R} \setminus \{0\}))$ such that \eqref{separation_condition_inequalities} is valid (i.e., $\mathcal{S}(C, K \mid \mathcal{C}_{{\rm BP}^*_2}(C, K)) \neq \emptyset$).
        \item[$4^\circ$]  There exists $(x^*, \alpha) \in (C^{aw\#} \cup -K^{aw\#})$ such that  \eqref{separation_condition_inequalities} is valid (i.e., $\mathcal{S}(C, K \mid \mathcal{C}_{{\rm BP}^*_3}(C, K)) \neq \emptyset$).
		
	\end{itemize}	
    Then, $1^\circ \Longrightarrow 2^\circ$ and $1^\circ \Longrightarrow 3^\circ \Longrightarrow 4^\circ$. 
    If one of the sets $\cl\, S_{C}$ and $\cl\, S_K$ is weakly compact, then $2^\circ \Longrightarrow 1^\circ$.
    If both $\cl\, S_{C}$ and $\cl\, S_{K}$ are weakly compact, then $4^\circ \Longrightarrow 1^\circ$.
\end{theorem}

\begin{proof} 
By Proposition \ref{prop:basic separation_BP_f}, we have $1^\circ \Longrightarrow 2^\circ$, while by Proposition \ref{prop:basic separation_BP_e} we have $1^\circ \Longleftrightarrow [0_X \notin \cl(S_{C} - S_K^0)$ or $0_X \notin \cl(S_{C}^0 - S_K)]$. Applying Theorems \ref{th:non-sym_sep_theorem_3} and \ref{th:non-sym_sep_theorem_4} we obtain $1^\circ \Longrightarrow 3^\circ \Longrightarrow 4^\circ$. The remaining implications (under weak compactness assumptions) follow from Proposition \ref{prop:basic separation_BP_f} and Theorems \ref{th:non-sym_sep_theorem_3} and \ref{th:non-sym_sep_theorem_4}. \qed
\end{proof}

A symmetric counterpart to the interpolation result in Theorem \ref{th:Omega1and2_BP_cone_nonsym} (for the non-symmetric approach) is given in the next theorem.

\begin{theorem} \label{th:Omega1and2_BP_cone_sym} 
If $0_X \notin \cl(S_{C} - S_{K})$ (in particular, $K \cap C = \{0_X\}$), then there exists $C_{\rm BP} \in\mathcal{C}_{{\rm BP}^*_1}(C)$ such that
\begin{equation}
\label{eq:Omega1Omega2_cond4}
K \cap C_{\rm BP} = \{0_X\} \quad \text{ and } \quad C \setminus \{0_X\} \subseteq \intt\, C_{\rm BP} \quad \text{ and } \quad C \subseteq C_{\rm BP},
\end{equation}
or there exists $C_{\rm BP} \in\mathcal{C}_{{\rm BP}^*_1}(K)$ such that
\begin{equation}
\label{eq:Omega1Omega2_cond5}
C \cap C_{\rm BP} = \{0_X\} \quad \text{ and } \quad K \setminus \{0_X\} \subseteq \intt\, C_{\rm BP} \quad \text{ and } \quad K \subseteq C_{\rm BP}.
\end{equation}
\end{theorem}

\begin{proof}
    By Proposition \ref{prop:basic separation_BP_e} the condition $0_X \notin \cl(S_{C} - S_{K})$ is equivalent to $0_X \notin \cl(S_{C} - S_{K}^0)$ or $0_X \notin \cl(S_{C}^0 - S_{K})$ (where the latter condition is equivalent to $0_X \notin \cl(S_{K} - S_{C}^0)$). Applying Theorem \ref{th:Omega1and2_BP_cone_nonsym} twice there exists $C_{\rm BP} \in \mathcal{C}_{{\rm BP}^*_1}(C)$ with \eqref{eq:Omega1Omega2_cond4} or there exists $C_{\rm BP} \in \mathcal{C}_{{\rm BP}^*_1}(K)$ with \eqref{eq:Omega1Omega2_cond5}. \qed
\end{proof}

\begin{remark}
We note that, taking into account Proposition \ref{prop:basic separation_BP_bd_1}, the conclusions of the previous theorem for the cones $C$ and $K$ also hold under the assumption of the separation property of their boundaries. More specifically, the condition $0_X \notin \cl(S_{C} - S_{K})$ can be replaced by the conditions $0_X \notin \cl(S_{\bd C} - S_{\bd K})$ and $K \cap C = \{0_X\}$.
\end{remark}

Under convexity assumptions concerning $C$ and $K$, we get the following symmetric cone separation result.

\begin{proposition} \label{prop:non-sym_sep_theorem_3}
Assume that $C$ and $K$ are convex. Consider the following assertions:
\begin{itemize}
    \item[$1^\circ$] $0_X \notin \cl(S_{C} - S_{K}) \cup (\cl\, S_{K}) \cup (\cl\,S_{C})$.
    \item[$2^\circ$] $0_X \notin \cl(S_{C} - S_K^0)\cup \cl(S_{C}^0 - S_K)$.
    \item[$3^\circ$] There exists $x^* \in X^* \setminus \{0_{X^*}\}$ such that \\$\sup_{k \in \cl\,S_{K}^0} x^*(k) = 0 < \inf_{c \in \cl\, S_{C}} x^*(c)$ and $\sup_{k \in \cl\,S_{K}} x^*(k) < 0 = \inf_{c \in \cl\, S_{C}^0} x^*(c)$.
    \item[$4^\circ$] There exists $x^* \in X^* \setminus \{0_{X^*}\}$ such that $\sup_{k \in \cl\,S_{K}} x^*(k) < 0 < \inf_{c \in \cl\, S_{C}} x^*(c)$.
    \item[$5^\circ$]
    $0_X \notin (\cl\, S_{K}) \cup (\cl\,S_{C})$ and there exists $x^* \in X^* \setminus \{0_{X^*}\}$ such that 
    \begin{equation*}
        x^*(k) > 0 > x^*(c) \quad \mbox{for all } k \in (\cl\,K)\setminus \{0_X\} \mbox{ and  }c \in (\cl\,C)\setminus \{0_X\}.
    \end{equation*}
    \item[$6^\circ$] $(\cl\,S_{K}^0) \cap (\cl\, S_{C})  = \emptyset$ and $(\cl\,S_{K}) \cap (\cl\, S_{C}^0)  = \emptyset$.
    \item[$7^\circ$] $(\cl\,S_{K}) \cap (\cl\, S_{C}) = \emptyset$
     and $0_X \notin (\cl\, S_{K}) \cup (\cl\,S_{C})$.
     \item[$8^\circ$] $(\cl\,K) \cap (\cl\, S_{C}) =  (\cl\,C) \cap (\cl\, S_{K}) = \emptyset$.
    \item[$9^\circ$] $(\cl\,K) \cap (\cl\,C) = \{0_X\}$ and $0_X \notin (\cl\, S_{K}) \cup (\cl\,S_{C})$.
\end{itemize}	
Then,  $4^\circ \Longleftrightarrow 3^\circ \Longrightarrow 2^\circ \Longrightarrow 1^\circ \Longrightarrow 7^\circ \Longleftrightarrow 6^\circ \Longleftrightarrow 8^\circ \Longleftrightarrow 9^\circ$,  as well as $4^\circ \Longrightarrow 5^\circ \Longrightarrow 9^\circ$. Moreover, if both of the sets $\cl\, S_{C}$ and $\cl\, S_K$ are weakly compact, then $6^\circ \Longrightarrow 2^\circ \Longrightarrow 4^\circ$ (i.e., assertions $1^\circ-9^\circ$ are equivalent).  
\end{proposition}

\begin{proof} The implication $2^\circ  \Longrightarrow 1^\circ$ is obvious.
Moreover, Proposition \ref{prop:basic separation_BP_f} shows $1^\circ \Longrightarrow 7^\circ$ while Proposition \ref{prop:basic separation_BP_b} provides $3^\circ \Longrightarrow 2^\circ$. Taking into account that $\mathbb{P} \cdot \cl\,S_{C} = \cl(\conv\,C) \setminus \{0_X\} = (\cl\,C) \setminus \{0_X\}$ and $\mathbb{P} \cdot \cl\,S_K = \cl(\conv\,K) \setminus \{0_X\} = (\cl\,K) \setminus \{0_X\}$  (by Lemma \ref{lem:clconvCandclSC}), we easily infer $4^\circ \Longrightarrow 5^\circ$ (where the strict inequalities in $4^\circ$ ensure $0_X \notin (\cl\, S_{K}) \cup (\cl\,S_{C})$). The implication $5^\circ \Longrightarrow 9^\circ$ is obvious. The equivalence of $3^\circ$ and $4^\circ$ is easy to check (note that $\sup_{k \in \cl\,S_{K}} x^*(k) \leq \sup_{k \in \cl\,S_{K}^0} x^*(k)$ and $\inf_{c \in \cl\, S_{C}^0} x^*(c) \leq \inf_{c \in \cl\, S_{C}} x^*(c)$). 
From Proposition \ref{prop:basic separation_BP_a} we get $6^\circ  \Longleftrightarrow 8^\circ \Longleftrightarrow 9^\circ$. 

Let us show the equivalence $7^\circ \Longleftrightarrow 8^\circ$. Of course, since $C$ and $K$ are convex cones, we have $(\cl\,S_K) \cup \{0_X\} \subseteq \cl\,K$ and $(\cl\,S_C) \cup \{0_X\} \subseteq \cl\,C$, hence $8^\circ \Longrightarrow 7^\circ$. The proof of the implication $7^\circ \Longrightarrow 8^\circ$ uses the same arguments (especially the convexity of the cones) as the proof of Proposition \ref{prop:basic separation_BP_f} ($3^\circ \Longrightarrow 4^\circ$).

For the last part of the proof, assume weak compactness of the sets $\cl\, S_{C}$ and $\cl\, S_K$. The implication $6^\circ \Longrightarrow 2^\circ$ is a consequence of 
Proposition \ref{prop:basic separation_BP_a}.
It remains to prove the implication $2^\circ \Longrightarrow 4^\circ$. Suppose that $2^\circ$ is valid. By Proposition \ref{prop:basic separation_BP_b} there exist $x^*, y^* \in X^* \setminus \{0_{X^*}\}$ such that $\sup_{k \in \cl\,S_{K}^0} x^*(k) = 0 < \inf_{c \in \cl\, S_{C}} x^*(c)$ and $\sup_{k \in \cl\,S_{K}} y^*(k) < 0 = \inf_{c \in \cl\, S_{C}^0} y^*(c)$.
Of course, then we have
$$\sup_{k \in \cl\,S_{K}} (x^*+y^*)(k) \leq \sup_{k \in \cl\,S_{K}^0} x^*(k) + \sup_{k \in \cl\,S_{K}} y^*(k) < 0$$
and
$$
0 < \inf_{c \in \cl\, S_{C}} x^*(c) + \inf_{c \in \cl\, S_{C}^0} y^*(c) \leq \inf_{c \in \cl\, S_{C}} (x^*+y^*)(c). 
$$
Thus, $4^\circ$ is true.
\qed
\end{proof}

\begin{remark}
The result in Proposition \ref{prop:non-sym_sep_theorem_3} extends the cone separation result (for non-trivial closed convex cones) mentioned in Proposition \ref{prop:linear_cone_separation}.
\end{remark}

Moreover, in the convex case, we are able to give  characterizations for
$$
\mathcal{N}(C, K \mid \mathcal{C}_{{\rm BP}^*_i}(C))) \neq \emptyset  \neq \mathcal{N}(K, C \mid \mathcal{C}_{{\rm BP}^*_i}(K))) \quad (i \in \{1,2,3\}),
$$
which also ensure
$$
\mathcal{S}(C, K \mid \mathcal{C}_{{\rm BP}^*_i}(C, K))) = \mathcal{N}(C, K \mid \mathcal{C}_{{\rm BP}^*_i}(C)))  \cup \mathcal{N}(K, C \mid \mathcal{C}_{{\rm BP}^*_i}(K))) \neq \emptyset. 
$$
\begin{corollary} \label{cor:non-sym_sep_theorem_3}
Assume that $C$ and $K$ are closed and convex. Consider the following assertions:
\begin{itemize}
    \item[$1^\circ$] $0_X \notin \cl(S_{C} - S_{K}) \cup (\cl\, S_{K}) \cup (\cl\,S_{C})$.
    \item[$2^\circ$] There exists $(x^*, \alpha) \in \cor\,C^{a+}$ such that \eqref{separation_condition_inequalities} is valid (i.e., $\mathcal{N}(C, K \mid \mathcal{C}_{{\rm BP}^*_1}(C)) \neq \emptyset$), and there exists $(x^*, \alpha) \in -\cor\,K^{a+}$ such that \eqref{separation_condition_inequalities} is valid (i.e., $\mathcal{N}(K, C \mid \mathcal{C}_{{\rm BP}^*_1}(K)) \neq \emptyset$).
    \item[$3^\circ$] There exists $(x^*, \alpha) \in C^{aw\#} \cap (X^* \times \mathbb{P})$ such that \eqref{separation_condition_inequalities} is valid (i.e., $\mathcal{N}(C, K \mid \mathcal{C}_{{\rm BP}^*_2}(C)) \neq \emptyset$), and there exists $(x^*, \alpha) \in -K^{aw\#} \cap (X^* \times -\mathbb{P})$ such that \eqref{separation_condition_inequalities} is valid (i.e., $\mathcal{N}(K, C \mid \mathcal{C}_{{\rm BP}^*_2}(K)) \neq \emptyset$).
    \item[$4^\circ$]  There exists $(x^*, \alpha) \in C^{aw\#}$ such that  \eqref{separation_condition_inequalities} is valid (i.e., $\mathcal{N}(C, K \mid \mathcal{C}_{{\rm BP}^*_3}(C)) \neq \emptyset$), and there exists $(x^*, \alpha) \in -K^{aw\#}$ such that  \eqref{separation_condition_inequalities} is valid (i.e., $\mathcal{N}(K, C \mid \mathcal{C}_{{\rm BP}^*_3}(K)) \neq \emptyset$).
\end{itemize}	
Then, $1^\circ \Longleftrightarrow 2^\circ \Longrightarrow 3^\circ \Longrightarrow 4^\circ$, and if the sets $\cl\, S_{C}$ and $\cl\, S_K$ are weakly compact, then $4^\circ \Longrightarrow 1^\circ$.
\end{corollary}

\begin{proof} By Proposition \ref{prop:non-sym_sep_theorem_3}, assertion $1^\circ$ is equivalent to $0_X \notin \cl(S_{C} - S_K^0) \cup \cl(S_{C}^0 - S_K)$.
Then, the result follows from Theorem \ref{th:non-sym_sep_theorem_1} ($1^\circ \Longleftrightarrow 4^\circ$), Theorem \ref{th:non-sym_sep_theorem_2} ($1^\circ \Longleftrightarrow 5^\circ$) and Theorems \ref{th:non-sym_sep_theorem_3} and \ref{th:non-sym_sep_theorem_4}. \qed
\end{proof}

Let us conclude this section with an example to show the similarities / differences of the non-symmetric and symmetric cone separation approaches. 

\begin{example} \label{ex:cone_separation}

Figure \ref{fig:cone_separation} (respectively, Figure \ref{fig:cone_separation2}) visualizes the non-symmetric (respectively, symmetric) cone separation approach for an example in the real normed space $(\mathbb{R}^2, ||\cdot||_2)$, where $||\cdot||_2$ denotes the Euclidean norm.

\begin{center}
\begin{figure}[h!]
\centering
\resizebox{1\hsize}{!}{\input{separation_cones_nonsym.TpX}}
\caption{Non-symmetric cone separation of two non-trivial, closed, pointed, solid cones $C$ and $K$ that satisfy $C \cap K = \{0_X\}$, ${\rm cl}\,S_{K}^0 = S_{K}^0$ and $0_X \notin {\rm cl}\, S_{C} = S_{C}$ in the real normed space $(\mathbb{R}^2, ||\cdot||_2)$: \\
(left image) $C$ is convex, $K$ is nonconvex, $(\cl\,S_{K}^0) \cap (\cl\, S_{C}) = \emptyset$ ($\Longleftrightarrow (\cl\,S_{\bd\,K}^0) \cap (\cl\, S_{\bd\,C}) = \emptyset$), $\mathcal{N}(C, K \mid \mathcal{C}_{{\rm BP}}) \neq \emptyset = \mathcal{N}(K, C \mid \mathcal{C}_{{\rm BP}})$;\\
(right image) $C$ is nonconvex, $K$ is convex, $(\cl\,S_{K}^0) \cap (\cl\, S_{C}) \neq \emptyset$ ($\Longleftrightarrow (\cl\,S_{\bd\,K}^0) \cap (\cl\, S_{\bd\,C}) \neq \emptyset$), $\mathcal{N}(C, K \mid \mathcal{C}_{{\rm BP}}) = \emptyset \neq \mathcal{N}(K, C \mid \mathcal{C}_{{\rm BP}})$.}
\label{fig:cone_separation}
\end{figure}
\begin{figure}[h!]
\centering
\resizebox{1\hsize}{!}{\input{separation_cones_sym.TpX}}
\caption{Symmetric cone separation of two non-trivial, closed, pointed, solid cones $C$ and $K$ that satisfy $C \cap K = \{0_X\}$, $0_X \notin {\rm cl}\,S_{C} = S_{C}$ and $0_X \notin {\rm cl}\, S_{K} = S_{K}$ in the real normed space $(\mathbb{R}^2, ||\cdot||_2)$: \\
(left image) $C$ is convex, $K$ is nonconvex, $(\cl\,S_{C}) \cap (\cl\, S_{K}) = \emptyset$ ($\Longleftrightarrow (\cl\,S_{\bd\,C}) \cap (\cl\, S_{\bd\,K}) = \emptyset$), $\mathcal{S}(C, K \mid \mathcal{C}_{{\rm BP}}) \neq \emptyset$;\\
(right image) $C$ is nonconvex, $K$ is convex, $(\cl\,S_{C}) \cap (\cl\, S_{K}) = \emptyset$ ($\Longleftrightarrow (\cl\,S_{\bd\,C}) \cap (\cl\, S_{\bd\,K}) = \emptyset$), $\mathcal{S}(C, K \mid \mathcal{C}_{{\rm BP}}) \neq \emptyset$.}
\label{fig:cone_separation2}
\end{figure}
\end{center}
\end{example}

\section{Existence of (Bounded) Convex Bases for Convex Cones} \label{sec: existence_bases}

In this section, we present some existence results for (bounded) bases of convex cones in the real normed space $(X, ||\cdot||)$. Assertions concerning the existence of a bounded base of convex cones play an important role in order-theoretical investigations, in functional analysis and optimization. Consider two non-trivial cones $C, D \subseteq X$ such that $\cl(\conv\, C) = \cl(\conv\,D)$.

\begin{remark} \label{rem:well_based_alternative}
    For any $D \in \{C, \cl\,C, \conv\,C, \cl(\conv\,C)\}$, we have $\cl(\conv\, C) = \cl(\conv\,D)$, and if $C$ is convex and acute, then the latter equality is valid for $D = \bd\,C$ (by Kasimbeyli \cite[Lem. 3.10]{Kasimbeyli2010}). 
\end{remark}

In the following theorem, we characterize the well-basedness of a non-trivial convex cone $C$ in terms of existence results of a non-trivial cone $K \subseteq X$ satisfying separation conditions involving the convex sets $S_{D}$ and $S_{\bd \, K}$ (respectively, $S_{\bd \, K}^0$ or $S_{K}^0$).

\begin{theorem}\label{th:well_based_alternative} Let $C$ and $D$ be two non-trivial cones in $X$ such that $\cl(\conv C)=\cl(\conv D)$. Then the following assertions are equivalent:
\begin{itemize}
\item[$1^\circ$] $\conv\, C$ (equivalently, $\cl(\conv\, C)$) is well-based. 
\item[$2^\circ$] There exists a functional $x^*\in X^*$ such that
\begin{equation}\label{eq1}
0 < \inf_{x \in S_{D}}\,x^*(x).
\end{equation}
\item[$3^\circ$] There exists a non-trivial cone $K\subseteq X$ such that $C\subseteq K$ and a functional $x^*\in X^*$ satisfying
\begin{equation*}\label{eq2}
0 < \sup_{k \in S_{\bd \, K}}\,x^*(k) \; \leq \; \inf_{x \in S_{D}}\,x^*(x).
\end{equation*}
\item[$4^\circ$] There exists a non-trivial cone $K\subseteq X$ such that $C\subseteq K$ and a functional $x^*\in X^*$ satisfying
\begin{equation}\label{eq3}
0 < \sup_{k \in S_{\bd \, K}}\,x^*(k) \; < \; \inf_{x \in S_{D}}\,x^*(x).
\end{equation}
\item[$5^\circ$] There exists a non-trivial cone $K\subseteq X$ with $C\subseteq K$ such that
\begin{equation}\label{eq4}
0_X\not \in \cl(S_D-S_{\bd\, K}^{0}).
\end{equation}
\item[$6^\circ$] There exists a non-trivial cone $K'\subseteq X$ such that
\begin{equation}\label{eq5}
0_X\not \in \cl(S_D-S_{K'}^{0}).
\end{equation}
\end{itemize}
Moreover, the value $\sup_{k \in S_{\bd \, K}}\,x^*(k)$ given in assertions $3^\circ$ and $4^\circ$ can be replaced by the value $\sup_{k \in S_{\bd \, K}^0}\,x^*(k)$, while the non-trivial cone $K$ given in assertions $3^\circ-5^\circ$ can be assumed to be solid, pointed, closed and convex.
If $\conv\, C = \conv\,D$, then the condition $C\subseteq K$ in assertions $4^\circ$ and $5^\circ$ can be replaced by the condition $C\setminus \{0_X\} \subseteq \intt\,K$ (i.e., $K$ is a dilating cone for $C$ (and for $\conv\, C$) if $K$ is convex).
\end{theorem}

\begin{proof}
    First, let us show the equivalence of the assertions $1^\circ-4^\circ$. 
    
    $1^\circ \Longleftrightarrow 2^\circ$: In view of \eqref{eq:KwellbasedIntK+neqEmpty}
    and 
    $
    C^{+} = (\cl C)^{+} = (\conv\, C)^{+} = (\cl(\conv\, C))^{+},
    $    
    for any non-trivial cone $D \subseteq X$ with $\cl(\conv\, C) = \cl(\conv\,D)$, we have
    \begin{align*}
        \exists\, x^* \in X^* \colon \eqref{eq1} \text{ is valid}
        & \iff 0_X \notin \cl\,S_D && \text{(by Proposition \ref{prop:basic separation})}\\
        & \iff \intt\,D^{+}\neq\emptyset\\
        & \iff \intt\,(\cl(\conv\, D))^{+}\neq\emptyset\\
        & \iff \intt\,(\cl(\conv\, C))^{+}\neq\emptyset\\
        & \iff \intt\,C^{+}\neq\emptyset\\
        & \iff\conv\,C \text{ is well-based} \\
        & \iff\cl(\conv\,C) \text{ is well-based}.
    \end{align*}
    This shows that $1^\circ \Longleftrightarrow 2^\circ$.

    $2^\circ \Longrightarrow 4^\circ$: Assume that \eqref{eq1} is valid for some $x^* \in X^*$. Define $\beta := \inf_{x \in S_{D}}\,x^*(x)$ and $K := C(x^*, \alpha)$ for some $\alpha \in (0, \beta)$. In view of Lemma~\ref{lem:properties_BP}~ ($3^\circ$) we have 
    $\alpha = \sup_{x \in S_{\bd \, K}} x^*(x)
     \; (= \sup_{x \in S_{\bd \, K}^0} x^*(x))$. 
    Note that the Bishop-Phelps cone $K$ is a non-trivial (solid, pointed, closed, convex) cone 
    taking into account that $||x^*||_* \geq \beta > \alpha > 0$. We conclude that \eqref{eq3} is valid.
    It remains to prove that $C \subseteq K$. Since $\beta > \alpha > 0$ we have $x^*(x) > \alpha > 0$ for all $x \in B_{D}$, hence $(x^*, \alpha) \in D^{a\#} \subseteq D^{a+}$.
    By our assumption $\cl(\conv\, C) = \cl(\conv\,D)$ we get 
    $
    (x^*, \alpha) \in D^{a+} = (\cl(\conv\,D))^{a+} = (\cl(\conv\, C))^{a+} = C^{a+},
    $
    which means that $C \subseteq K$. Note that, if $\conv\, C = \conv\,D$, then 
    $(x^*, \alpha) \in D^{a\#} = (\conv\, D)^{a\#} = (\conv\, C)^{a\#} = C^{a\#},$
    hence $C \setminus \{0_X\} \subseteq C^>(x^*, \alpha) = \intt\,K$.  This shows that $4^\circ$ is true. 
    
    $4^\circ \Longrightarrow 3^\circ \Longrightarrow 2^\circ$: Both implications are obvious.

    \medskip
     
    From the above proof of the equivalence of $1^\circ-4^\circ$ it can be seen that in the statements $3^\circ$ and $4^\circ$ the value $\sup_{k \in S_{\bd \, K}}\,x^*(k)$ can be replaced by the value $\sup_{k \in S_{\bd \, K}^0}\,x^*(k)$ (having in mind Lemma \ref{lem:properties_BP} ($3^\circ$)), while the non-trivial cone $K$ can be assumed to be solid, pointed, closed and convex.

    \medskip
    Now, let us show the equivalence of the assertions $2^\circ$ (equivalently, $4^\circ$), $5^\circ$ and $6^\circ$.

    $4^\circ \Longrightarrow (5^\circ \wedge 6^\circ)$: Assume that $4^\circ$ with $S_{\bd \, K}$ replaced by $S_{\bd \, K}^0$ in \eqref{eq3} is valid. As observed above, we can suppose that the non-trivial cone $K$ is (solid, pointed) closed and convex.  By Proposition \ref{prop:basic separation_BP_a} (applied for $D$ and $\bd \, K$ in the role $C$ and $K$)  we immediately get \eqref{eq4}, hence $5^\circ$ is valid. 
    Define the non-trivial cone $K' := (X\setminus K) \cup \{0_X\}$. Then, 
    $\bd\,K = \bd\,K'$, and since $K \supseteq \cl(\conv\, C) = \cl(\conv\,D) \supseteq D$, we have $D \cap K' = \{0_X\}$. The latter condition combined with \eqref{eq4} is equivalent to \eqref{eq5} in view of Proposition \ref{prop:basic separation_BP_bd} (applied for $D$ and $K'$ in the role $C$ and $K$). This proves that $6^\circ$ is valid.

    $(5^\circ \vee 6^\circ) \Longrightarrow 2^\circ$: 
    Since $0_X \in S_{K'}^{0} \cap S_{\bd K}^{0}$, we have that $0_X\not \in \cl(S_D - S_{\bd K}^{0})$ (respectively, $0_X\not \in \cl(S_D - S_{K'}^{0})$) implies $0_X\not \in \cl(S_D -\{0_X\}) = \cl\, S_D$, which is equivalent to $2^\circ$.
    
    \qed
    
\end{proof}

\begin{remark} \label{rem:well_based_alternative_bis}
    Note that the results in Sections \ref{sec:sym_and_non-sym_separation} and \ref{sec:sym_and_non-sym_cone_separation} provide some more equivalent statements for $6^\circ$ in Theorem \ref{th:well_based_alternative}. Indeed, by Theorem \ref{th:non-sym_sep_theorem_1} we have
    $$
    \eqref{eq5} \text{ is valid } \iff \mathcal{N}(D, K' \mid \mathcal{C}_{{\rm BP}^*_1}(D)) \neq \emptyset,
    $$
    and by Theorem \ref{th:non-sym_sep_theorem_3}, if $\cl\, S_{D}$ is weakly compact, then 
    \begin{align*}
        \eqref{eq5} \text{ is valid } &\; \iff \mathcal{N}(D, K' \mid \mathcal{C}_{{\rm BP}^*_2}(D)) \neq \emptyset\\
        &\; \iff \mathcal{N}(D, K' \mid \mathcal{C}_{{\rm BP}^*_3}(D)) \neq \emptyset \\
        &\; \iff (\cl\,S_{K'}^0) \cap (\cl\, S_{D}) = \emptyset.
    \end{align*}
\end{remark}

The following result in Corollary \ref{cor:well_based_alternative} highlights the differences and similarities between a convex cone $\conv\,C$ having a (bounded) convex base and a convex closed cone $\cl(\conv\,C)$ having a (bounded) convex base.  This is achieved by giving equivalent conditions based on inclusions between the (not necessarily convex) norm-bases $B_{C}$ and $B_{\cl\,C}$ of the (not necessarily convex) cones $C$ and $\cl\,C$, and certain supporting half-spaces.

\begin{corollary} \label{cor:well_based_alternative}
The following assertions hold:  
\begin{itemize}
    \item[$1^\circ$] $\conv\,C$ has a convex base if and only if there exists $x^*\in X^*$ such that 
    $$B_C \subseteq \{x \in X \mid x^*(x) > 0\}.$$
     \item[$2^\circ$] If $\conv(\cl\,C) = \cl(\conv\,C)$ (e.g., if $C$ is convex), then\\$\cl(\conv\,C)$ has a convex base if and only if there exists $x^*\in X^*$ such that $$B_{\cl \,C} \subseteq \{x \in X \mid x^*(x) > 0\}.$$
    \item[$3^\circ$]  $\conv\,C$ is well-based if and only if there exist $x^*\in X^*$ and $\alpha>0$ such that $$ B_{C} \subseteq \{x \in X \mid x^*(x) \geq \alpha\}.$$
    \item[$4^\circ$] $\cl(\conv\,C)$ is well-based if and only if there exist $x^*\in X^*$ and $\alpha>0$ such that $$B_{\cl \,C} \subseteq \{x \in X \mid x^*(x) \geq \alpha\}.$$
\end{itemize}
\end{corollary}

\begin{proof}
  In view of \eqref{eq:K+=convK+=clvonvK+} and \eqref{eq:KbasedKSharpneqEmpty}, we have
\begin{align*}
   \conv\,C \text{ has a convex base} 
    & \iff (\conv\,C)^\# \neq \emptyset \\
    & \iff C^\# \neq \emptyset \\
    & \iff \exists\, x^* \in X^*:\; x^*(x) > 0 \text{ for all } x \in B_{C}
\end{align*}
and
\begin{align*}
    \cl(\conv\,C) \text{ has a convex base} 
    & \iff (\cl(\conv\,C))^\# \neq \emptyset \\
    & \iff (\conv(\cl\,C))^\# \neq \emptyset \\
    & \iff (\cl\,C)^\# \neq \emptyset \\
    & \iff \exists\, x^* \in X^*:\; x^*(x) > 0 \text{ for all } x \in B_{\cl\,C}.
\end{align*}
Moreover, in view of \eqref{eq:KwellbasedIntK+neqEmpty}, \eqref{eq:clconvCa+=Ca+} and Proposition \ref{prop:properties_augmented_dual_cone} ($1^\circ$), it follows that
\begin{align*}
    \conv\,C \text{ is well-based} 
    & \iff C^{a+} \cap (X^* \times \mathbb{P}) \neq \emptyset \\
    & \iff \exists\, (x^*, \alpha) \in X^* \times \mathbb{P}:\; x^*(x) \geq \alpha \text{ for all } x \in B_{C}
\end{align*}
and
\begin{align*}
\cl(\conv\,C) \text{ is well-based} 
    & \iff C^{a+} \cap (X^* \times \mathbb{P}) \neq \emptyset \\
    & \iff (\cl\,C)^{a+} \cap (X^* \times \mathbb{P}) \neq \emptyset \\
    & \iff \exists\, (x^*, \alpha)  \in X^* \times \mathbb{P}:\; x^*(x) \geq \alpha \text{ for all } x \in B_{\cl\,C}.
\end{align*}
\qed
\end{proof}

\begin{remark}\label{r-appl}
The existence results for a (bounded) convex base of convex cones in real normed spaces which we have derived in this section employing our (symmetric and non-symmetric) cone separation results are useful for various applications. We will briefly mention some important applications of our results in the field of optimization.
Especially, the existence of a bounded base is important for showing existence and density results in optimization and functional analysis. In the following, we list some results from the literature where cones with a bounded base play a role.
\begin{itemize}
\item[$\bullet$] In the minimal point theorem by Phelps, the involved cone has a bounded base (compare \cite[Proposition 3.11.2]{GoeRiaTamZal2023} and the corresponding discussion there). 

\item[$\bullet$] In the discussion of sufficient conditions for the existence of minimal points of subsets of product spaces, the boundedness of the base of a convex cone is important (see \cite[Section 3.11.1]{GoeRiaTamZal2023}). 

\item[$\bullet$] Postolic\u{a} has shown in \cite[Corollary 3.2.1]{Pos93} that a bounded and closed subset of a Hausdorff locally convex space has the domination property under the assumption that the involved cone has a complete bounded base. The domination property plays an important role in optimization theory.

\item[$\bullet$] Petschke in \cite{Petschke1990} obtains a density result for the set of positive proper minimals within the set of minimal points, extending the well-known Arrow-Barankin-Blackwell theorem to infinite dimension spaces under the assumption that the ordering cone is well-based. This result was later attempted to be generalized by Gong in \cite{Gong1995} using the notion of a point of continuity (a slight weakening of the notion of a denting point). Whether the result established by Gong was indeed a generalization of Petschke's result was stated as an open problem in \cite{Gong1995}, which was later analyzed in \cite[Section 4]{GarciaCastano2021}.

\end{itemize}
\end{remark}
\color{black}

\section{Conclusions} 
\label{sec:conclusions}

From a classical separation perspective in convex analysis and optimization (without considering a specific application), a symmetric separation concept for cones seems to be preferable. However, it is important to note that many significant new results for the symmetric cone separation approach were derived from the non-symmetric cone separation approach (thanks to Propositions~\ref{prop:basic separation_BP_e} and~\ref{prop:basic separation_BP_f}). As discussed in Remark~\ref{r-appl}, our results on the existence of a (bounded) base of convex cones have important applications in optimization, where they can be used to derive existence statements, for example, in the proof of general minimal point theorems.

Future research directions include extending the symmetric cone separation results to separation theorems for (not necessarily convex) sets without cone properties in real normed spaces. Additionally, we aim to develop new scalarization methods based on Bishop-Phelps cone-representing scalarizing functions, such as norm-linear functions, for vector optimization. These scalarizing functions are expected to be useful for establishing existence and density results for properly minimal points in vector optimization problems.

\section*{Acknowledgement}
The authors Fernando García-Castaño and Miguel Ángel Melguizo-Padial have been supported by Project PID2021-122126NB-C32 funded by MICIU/AEI /10.13039/501100011033/ and FEDER A way of making Europe. The authors would also like to thank the anonymous reviewers and the editor for their valuable comments and suggestions, which helped to improve the quality of the manuscript.



\end{document}